\documentclass{amsart}

\usepackage{amsmath}
\usepackage{graphicx}
\usepackage{amsfonts}
\usepackage{amssymb}

\newtheorem{theorem}{Theorem}
\theoremstyle{plain}

\newtheorem{example}{Example}

\newtheorem{proposition}{Proposition}
\newtheorem{remark}{Remark}

\numberwithin{equation}{section}

\begin{document}

\title[Secondary cohomology operations]
{Secondary cohomology operations
and the loop space cohomology}

\author{Samson Saneblidze}
\address{Samson Saneblidze, A. Razmadze Mathematical Institute,
I.Javakhishvili Tbilisi State University  2, Merab Aleksidze II Lane,
 Tbilisi 0193, Georgia}
 \email{sane@rmi.ge}

\thanks{This work was supported by Shota Rustaveli
National Science Foundation of Georgia (SRNSFG) [grant number
FR-23-5538]}


\date{}

\begin{abstract}

Motivated by the loop space cohomology we construct the secondary
operations on the cohomology $H^*(X; \mathbb{Z}_p)$  to be  a Hopf
algebra
for a simply connected space $X.$
The  loop space cohomology ring $H^*(\Omega X; \mathbb{Z}_p)$ is
calculated in terms of generators and relations. This answers to A.
Borel's decomposition of a Hopf algebra into a tensor product of the
monogenic ones in which the heights of generators are determined
by means of the action of the primary and secondary cohomology
operations on $H^*(X;\mathbb{Z}_p).$ An application for calculating
of the loop space cohomology ring of the exceptional group  $F_4$ is
given.

\end{abstract}

\maketitle

\section{Introduction}

Let $X$ be a simply connected topological space and
$ H^*(X; \mathbb{Z}_p)$ be  the cohomology algebra in coefficients
$\mathbb{Z}_p=\mathbb{Z}/p\mathbb{Z}$ where   $\mathbb{Z}$ is the
integers and   $p$ is a prime.

Given $n\geq1,$ let $P^*_n(X)\subset {H}^*(X;\mathbb{Z}_p)$ be the
subset of elements of finite height
\[P^{*}_n(X)=\{x\in {H}^* (X;\mathbb{Z}_p)\mid  x^{n+1}=0,\,n\geq 1 \},
\]
where for  an odd $p$ we have  $P^{od}_{1}(X)=H^{od}(X;\mathbb{Z}_p)$
in odd degrees.
 The subset of elements of infinite height is denoted by
 $P^*_{\infty}(X)\subset {H}^*(X;\mathbb{Z}_p).$

 Let $\mathcal{P}_1$ be the primary cohomology operation
\[ \mathcal{P}_1: H^m (X;\mathbb Z_p)  \rightarrow H^{(m-1)p+1}(X;\mathbb
Z_p),\ \ \
[c]\rightarrow \left[ c^{\smile_1 p}\right],
\]
where
  $c\in C^m(X;\mathbb Z_p)$   is a cocycle  and
  $c^{\smile_1 p}=c\smile_1\cdots \smile_1 c$ is a $(p-1)$ -- iteration
 of the Steenrod cochain $c\smile_1c$ -- operation.
For $r,n\geq 1$ we introduce  the secondary
cohomology operations
\begin{equation}\label{secondary}
\begin{array}{lll}
\psi_{r,1}:P^m_1 (X)\ \ \rightarrow H^{(m-1)p^{r+1}+1}(X;\mathbb
Z_p)/\operatorname{Im}\mathcal{P}_1,
\vspace{1mm}\\
\psi_{r,n}: P^m_{n>1}(X)\rightarrow H^{(m(n+1)-2)p^r+1 }
(X;\mathbb Z_p)/ \operatorname{Im}\mathcal{P}_1,
\end{array}
 \end{equation}
   in which $\psi_{1,p^k-1}=\psi_k$ is the Adams secondary cohomology
   operation for
$p$ odd or $p=2$ and $k>1$ (cf. \cite{adams-hopf1}, \cite{harper}, \cite{kraines2},   \cite{tangora}).
Note that the maps $\psi_{r,n}$ are
  linear for $n+1=p^k,\,k\geq 1.$

Roughly speaking, as the homotopy commutativity of the cup  product on\\
$ C^\ast(X;\mathbb Z_p)$ gives rise to
the primary cohomology operations, here the homotopy multiplicative
map between the bar-constructions $B H^\ast(X;\mathbb Z_p)\rightarrow B
C^\ast(X;\mathbb Z_p)$
induced by a cocycle-choosing map
 $ H^\ast(X;\mathbb Z_p)\rightarrow C^\ast(X;\mathbb Z_p)$
 gives rise to
the above secondary cohomology operations, where
   $B H^\ast(X;\mathbb Z_p)$   and $B C^\ast(X;\mathbb Z_p)$ are endowed by
   the standard   shuffle (commutative) product
and the standard (geometric, non-commutative) product respectively.

The secondary cohomology operations are related with the loop
cohomology ring
$H^\ast(\Omega X; \mathbb{Z}_p)$ as follows. Let $\Omega X$ be the
(based) loop space on $X,$ and let
\[
\sigma: H^m(X;\mathbb{Z}_p)\rightarrow H^{m-1}(\Omega X; \mathbb{Z}_p)
 \]
 be the loop suspension homomorphism.
Given $z\in H_\ast(\Omega X; \mathbb{Z}_p)$  and $m\geq 3,$ let the
symmetric Massey product of $z$ be defined and denoted by
$\left< z\right>^{m}$
\cite{kraines} (when
$X$ is itself an H-space, $\left< z\right>^{m}$ consists  of a
unique element; cf.\,\cite{kochman}).
For example, if $z=(\sigma x)^{\#}$ is the dual element of $\sigma x\neq 0,$
then   $\left< z\right>^{n+1}$ is defined and
is non-zero (when $n=1, \left< z\right>^{2} =z^2$). Define
an element $\omega_{1,n}(x)\in H^*(\Omega X; \mathbb{Z}_p)$ as the dual
element:
 \[   \omega_{1,n}(x)=  \omega^\#  \ \ \text{with}\ \
 \omega\in   \left< z\right>^{n+1},\ \   z=(\sigma x)^{\#},\ \ \sigma x\neq 0.
  \]
    Let  $\Delta: H^*(\Omega X; \mathbb{Z}_p)\rightarrow H^*(\Omega X;
    \mathbb{Z}_p) \otimes H^*(\Omega X; \mathbb{Z}_p)$
    be the coproduct and $\widetilde \Delta$ be  $\Delta$ without the
    primitive terms. 

 For $x\in P^*_n(X),$ 
    let  $\ell_x$  be the maximal integer with  $ 0\leq \ell_x\leq  \infty  $
  such that there exists
 a string of elements of $H^*(\Omega X; \mathbb{Z}_p)$
\begin{equation}\label{string}
(\omega_{0,n}(x),\, \omega_{1,n}(x),\, \omega_{2,n}(x),...,
\omega_{\ell_x,n}(x))\ \  \text{with}\ \    \omega_{0,n}(x):=\sigma x
 \end{equation}
satisfying  for $r\geq 1$ the  equalities
\begin{equation}\label{coprod}
\hspace{0.3in}
\widetilde \Delta(\omega_{r,n}(x))=\sum_
{\substack{0\leq i_1<\cdots < i_s < r\\ 0\leq j_1<\cdots <j_t<r}}\!\!\!
\omega_{i_1,n}\cdots\omega_{i_s,n}\otimes 
\omega_{j_1,n}\cdots\omega_{j_t,n}+  a'_r\otimes a''_r,
\end{equation}
 where  monomials are taken with respect to the product in  $H^*(\Omega X; \mathbb{Z}_p)$  with  $i_1,j_1>0$ for $n>1$  or for $\sigma x=0,$
   while a component $a'_r\otimes a''_r$ does not contain $\omega_{i,n}$'s in the both sides simultaneously;  
  in particular, 
$\widetilde \Delta(\omega_{1,n}(x))=0$ for $n>1.$

 The string may be infinite (i.e., $\ell_x=\infty$). The  construction
 of $\omega_{r,n}(x)$ is given in Section \ref{2operations}.
Let $\mathcal{H}^*_p(X)\subset H^*(X;\mathbb{Z}_p)$ be a subset of
 multiplicative generators, and let 
 \[   \mathcal {S}^*(X)  =\{x\in \mathcal{H}^*_p(X)\mid \sigma x\neq 0 \}\ \
 \text{with}\ \  \mathcal {S}_n^*(X)= \mathcal {S}^*(X)\cap P^*_n(X).
\]
 Denote
 \[\hspace{-1.22in}\Phi_0(X)=\{ \sigma x \mid x\in 
\mathcal{S}_n(X)\setminus \operatorname{Im} \mathcal{P}_1,\, 1\leq n\leq \infty \},\]
\[\Phi_n(X)=\{ \left(\, \omega_{1,n}(x),..., \omega_{\ell_x,n}(x)\right) \mid x\in
  P_{n}(X)\cap\mathcal{H}_p(X),\, 1\leq n < \infty 
\}.\]
Let $\mathcal{P}_1^{(m)}$ denote the $m$-fold composition
$\mathcal{P}_1\circ \cdots \circ \mathcal{P}_1,$
and let \[\mathcal{D}_p^*(X):=H^+(X;\mathbb Z_p)\cdot H^+(X;\mathbb
Z_p)\subset H^*(X;\mathbb Z_p)\] be the decomposables.
 For
 $w\in H^*(\Omega X; \mathbb{Z}_p)$ and $q\geq 1,$ denote by $w^q\in
 H^*(\Omega X; \mathbb{Z}_p)$ the $q$ - th  power of $w.$
 The expression $ f(\psi_{*,*}(x))$ means that a function $f$ is evaluated
 on a representative of $\psi_{*,*}(x).$ Also fix the convention 
 \begin{equation} \label{psi0}
\psi_{0,n}(x)= \begin{cases}\mathcal{P}_1(x),& n=1,\\   
0, & \text{otherwise}.
\end{cases}
\end{equation}
\begin{theorem}\label{premain}
Let $X$ be a simply connected topological space such that the
cohomology ring $H^*(X; \mathbb{Z}_p)$ is a Hopf algebra. Then
$\Phi_0(X)\cup \Phi_n(X)$ is
the set of  multiplicative generators  of
      $ H^*(\Omega X;\mathbb{Z}_p),$  $p\geq 2,$ and

      (i)  For  $k\geq 1$ and  $\sigma x\in \Phi_0(X)$
       unless $x\in \mathcal{H}^{ev}_{p>2}(X):$
\[ (\sigma x)^{\,p^k}=\sigma \mathcal{P}_1^{(k)}(x); \]

(ii) For $r,n,k\geq1$   and $\omega_{r,n}(x)\in \Phi_n(X):$

\begin{equation*} 
 \omega_{r,n}(x)^{\,p^k}= 
\sum_{\substack{i+j=r\\ i, j\geq 0}} \!\!   \omega_{i,n}(\psi_{j,n}(x))^{\,p^{k-1}} .
\end{equation*}

\end{theorem}
From Theorem \ref{premain} immediately follows

\begin{theorem}\label{main}
  Let $X$ be a simply connected topological space such that the
  cohomology ring  $H^*(X; \mathbb{Z}_p)$ is a Hopf algebra.

(i)  For  $k\geq 1$ and  $\sigma x\in \Phi_0(X) ,$
 let $\nu_0(x)\geq1$ be the smallest integer such that  $
 \mathcal{P}^{(\nu_0(x))}_1 (x)\in \mathcal{D}^*_p(X)$ unless $ x\in
 \mathcal{H}^{ev}_{p>2}(X).$
 Then
\[ (\sigma x)^{p^{\nu_0(x)}}=0; \]

(ii) For $r,n,k
\geq1$   and $x\in P^*_n(X),$  let    $\nu_{j,n}(x)\geq1$
be the smallest integer     with  $0\leq j\leq r$ such that

 \[  \mathcal{P}_1^{(\nu_{j,n}(x)-1)}\left(\psi_{j,n}(x)\right)      \in \mathcal{D}^*_p(X)\ \   \text{for all}\ \  j.  \]

\noindent
       Then $p^{\nu_{r,n}(x)}$ is the height of the generator
     $\omega_{r,n}(x)\in \Phi_n(X)$
       \[
\omega_{r,n}(x)^{\,p^{\nu_{r,n}(x)}}=0.\]

\end{theorem}

A particular case of the theorem for $p=2$ and $n=\infty$   is proved
in \cite{sanePOL}.
Furthermore, we have an explicit description of the set $\mathcal{S}^*_n(X)$
in terms  of the operations $\mathcal{P}_1$ and  $\psi_{1,n}$  and
higher order Bockstein homomorphisms $\beta_k$
associated with the short exact  sequence
\[
0\rightarrow \mathbb{Z}_p\rightarrow \mathbb{Z}_{p^{k+1}} \rightarrow
\mathbb{Z}_{p^k}\rightarrow 0
\]
in the following  theorem being in fact a reformulation of the result
in   \cite{kraines2}
(cf. also  \cite{clark}, \cite{kochman}).
Let
\begin{equation}\label{ex} \epsilon_x=\begin{cases}
      0, & x\in P^{od}_1(X)\ \ \text{and}\ \  p>2,     \\
         1, &  \text{otherwise} .
       \end{cases}
       \end{equation}
         
\begin{theorem}\label{sigmazero}
Let   $\mathcal{I}^*_p (X)= \mathcal{I}^*_{p,1}(X)\cup
\mathcal{I}^*_{p,2}(X) \subset \mathcal{H}^*_p(X)$
be the subset defined   for $\ell_x \geq \epsilon_x $   by
\begin{equation*}
\begin{array}{ll}
 \mathcal{I}^*_{p,1} (X)\!=\!
 \left\{y\in \mathcal{H}^\ast_p(X) \mid\!
  y=\beta_{_{\ell_x+1}}\mathcal{P}^{(\ell_x+1)}_1(x), \hspace{0.37in}
  x\in \mathcal{H}_p(X)\cap P^{od}_1(X)\right\},
  \vspace{1mm}\\
 \mathcal{I}^*_{p,2}(X)\!= \!\left\{y\in \mathcal{H}^\ast_p(X) \mid
 \! y=\beta_{_{\ell_x}}\!\mathcal{P}^{(\ell_x-1)}_1\!\psi_{1,n}(x),\,\,  x\in
  \mathcal{H}_p(X)\cap (P^\ast_{n}(X)\!\setminus\! P^{od}_1(X))\!\right\}.
  \end{array}
\end{equation*}
Then
 $\operatorname{Ker}\sigma=\mathcal{I}^*_p(X)\cup\mathcal{D}^*_p(X);$
hence,
$\mathcal{S}^*_n(X)=(P^*_n(X)  \setminus \mathcal{I}^*_p(X))\cap\mathcal{H}_p^*(X) .
$
\end{theorem}

The method of calculating of the  loop space cohomology relies on the
integral filtered Hirsch model
of a space $X$
\[\varphi: (R_aH^\ast(X),d_h)\rightarrow C^*(X;\mathbb{Z}) \]
constructed in \cite{saneFiltered} (cf.\,\cite{hal-sta},
\cite{hueb-kade}). However,  the   mod p  filtered Hirsch  algebra
\[ (RH^\ast(X),d_h)\otimes \mathbb{Z}_p\]
for the minimal Hirsch resolution $RH^\ast(X)$ of $H^\ast(X)$
 described here has a quite simple form when   $H^\ast(X; \mathbb Z_p)$ is
 a Hopf algebra.
The fact that $ RH^\ast(X)$ is a free $\mathbb{Z}$-algebra
 enables to control  the higher order Bockstein maps $\beta_k$ to
 establish all necessary formulas therein (compare \cite{browder}), and
 then to pass to mod p coefficients.

The paper is organized as follows. Section 2 reviews the filtered  Hirsch models.
In Section 3 the secondary cohomology operations are constructed. In Section 4 Theorems 1 and 3 are proved. The mod 2 and 3 loop cohomology ring of the exceptional Lie group $F_4$ is calculated in Section 5.

\section{The Hirsch filtered models of a space}\label{shirsch}
Given a commutative graded $\mathbb{Z}$\,-\,algebra (cga) $H^\ast$
there are two kinds of  Hirsch resolutions of $H^*:$
\[\rho_a:(R_aH^\ast,d)\rightarrow H^\ast,  \ \ \text{the absolute Hirsch resolution }    \]
and
\[
\rho:(RH^\ast,d)\rightarrow H^\ast, \ \  \text{the minimal Hirsch
resolution}.    \]
 In particular, the both resolutions also serve as
free additive resolutions of the graded group $H^\ast$  (but not
necessarily short ones!),  so that both $R_aH^\ast$  and $RH^\ast$ are
free groups additively.
The absolute Hirsch resolution $R_aH,$ beside  the Steenrod cochain
operation $E_{1,1}=\, \smile_1 ,$ the cup-one product, that
 measures  the non-commutativity of the cup-product,  is endowed with
 the higher
order operations\\
 $E_{p,q},\,p,q\geq 1$  as they usually exist in the
(simplicial, cubical) cochain complex $C^*(X;\mathbb Z).$ These
operations  appear  to measure the deviations of the cup-one product
from left
and right derivation with respect to  the cup-product, and so on.
The minimal Hirsch resolution $RH^\ast$ is in fact endowed     with
only  binary operation $E_{1,1}$ related
 with the cup-product  by explicit formulas (cf.
 (\ref{hirsch1})--(\ref{hirsch2}) below). Consequently,
 there is no immediate map  $(RH^\ast,d_h)\rightarrow
 C^*(X;\mathbb{Z})$
 of Hirsch algebras   for
 $H^\ast:=H^\ast(X;\mathbb Z) ,$ instead there are the
   zig-zag Hirsch algebra maps
\begin{equation}\label{zig-zag}
 (RH^\ast,d_h) \overset {\varrho}{\longleftarrow}
 (R_aH^\ast,d_h)\overset{\varphi}{\longrightarrow} C^*(X;\mathbb{Z}).
 \end{equation}
 In fact, $RH^\ast=R_aH^\ast/J$ for a certain Hirsch ideal $J\subset
 R_aH^\ast,$ so that the freeness of the multiplicative structure in
 $RH^*$ is kept. Then
 \[  (RH^\ast\otimes \mathbb{Z}_p,d_h) \overset{\varrho\otimes
 1}{\longleftarrow}  (R_aH^\ast\otimes \mathbb{Z}_p,
 d_h)\overset{\varphi\otimes 1}{\longrightarrow} C^*(X;\mathbb Z_p) \]
 is  regarded as the mod p  filtered  Hirsch  model of $X.$

\subsection{A minimal complex to calculate $H^*(\Omega X;\mathbb{Z}_p)$
}

For a $\Bbbk$ -- module $A,$ let $T(A)=\bigoplus_{i=0}^{\infty }A^{\otimes i}$
with $
A^{0}=\Bbbk $ be the tensor module of $A$. An element $a_{1}\otimes
\cdots
\otimes a_{n}\in A^{\otimes n}$ is denoted by $[a_{1}|\cdots |a_{n}]$
when $%
T(A)$ is viewed as the tensor coalgebra or by $a_{1}\cdots a_{n}$ when
$T(A)$
is viewed as the tensor algebra. We denote by $s^{-1}A$ the
desuspension of $%
A$, i.e., $(s^{-1}A)^{i}=A^{i+1}$.

A dga $(A,d_{A})$ is assumed to be supplemented: it has
the
form $A=\widetilde{A}\oplus \Bbbk .$ The (reduced) bar-construction $BA$ on
$A$
is the tensor coalgebra $T(\bar{A}),\ \bar{A}=s^{-1}\widetilde{A},$ with
differential $d=d_{1}+d_{2}$ given for $[\bar{a}_{1}|\dotsb
|\bar{a}_{n}]\in
T^{n}(\bar{A})$ by
\begin{equation*}
d_{1}[\bar{a}_{1}|\dotsb |\bar{a}_{n}]=-\sum_{1\leq i\leq
n}(-1)^{\epsilon
_{i-1}^{a}}[\bar{a}_{1}|\dotsb |\overline{d_{A}(a_{i})}|\dotsb
|\bar{a}_{n}]
\end{equation*}%
and
\begin{equation*}
d_{2}[\bar{a}_{1}|\dotsb |\bar{a}_{n}]=-\sum_{1\leq i<n}(-1)^{\epsilon
_{i}^{a}}[\bar{a}_{1}|\dotsb |\overline{a_{i}a_{i+1}}|\dotsb
|\bar{a}_{n}],
\end{equation*}%
where $\epsilon _{i}^{x}=|x_{1}|+\cdots +|{x_{i}}|+i.$

  We have  $RH=T(V)$ for some $V.$     Denote
  $\overline{V}=s^{-1}(V^{>0})\oplus
{\mathbb Z} ,$ and form
the cochain complex $(\,\overline{V},\bar{d}_h)$
with
 the differential
 $\bar {d}_h:=\bar d+\bar h$ on $\overline{V}$  obtained     by the
 restriction of $d_h=d+h$ to $V.$
The bar-construction $(B(RH), \mu_e)$ is a dga with the product $\mu_e$
defined  by the binary operation
 $e:=E_{1,1}$  on $RH$
  (cf. \cite{saneFiltered}).
Let  \[\phi: B(RH)\rightarrow \overline{RH} \rightarrow \overline{V}\]
 be the standard projection of cochain
complexes. Convert $\phi$ as a map of dga's by introducing  a product
on   $\overline{V}.$ Namely,  for $\bar a, \bar b\in \overline{V}$
define
 \[
 \bar a\bar b=\overline {a\smile_1b} \ \ \  \text{with}\ \ \
  \bar a1=1\bar a=\bar a.
   \]
 Then we get  the following sequence of algebra isomorphisms
\begin{multline*}
 H^\ast\!\left(BC^{\ast}(X;\mathbb Z),d_{_{BC}},\mu_{_{E}}\right)
 \underset{\approx}{\xleftarrow{B\varphi^\ast} }
 H^\ast\!\left(B(R_aH),d_{_{B(R_aH)}},\mu_{_{E}}\right)
 \underset{\approx}{\xrightarrow{B\varrho^\ast}}\\
H^\ast\!\left(B(RH),d_{_{B(RH)}},{\mu}_{e}\right)
\underset{\approx}{\xrightarrow{\phi^*}}
H^\ast\!\left(\,\overline{V}, \bar d_h\right),
\end{multline*}
where $\mu_{_E}$ denotes the products on $BC(X;\mathbb Z)$ and $B(R_aH)$
defined by the Hirsch structural operations $E:=\{E_{p,q}\}_{p,q\geq
1}$
on  $C^*(X;\mathbb Z)$ and $R_aH$ respectively (cf. \cite{saneFiltered}).
 Thus  we get the algebra isomorphisms (compare \cite{F-H-T})
 \begin{equation*}
H^{\ast }(\,\overline{V},\bar{d}_{h})\approx H^{\ast }(BC^{\ast }
(X;\mathbb Z),d_{_{BC}})\approx H^{\ast }(\Omega X;\mathbb Z ).
\end{equation*}
 Consequently,
 the calculation of the multiplicative structure of
$H^*(\Omega X;\mathbb{Z}_p)$ reduces to that of
$H^*(\,\overline{V},\bar d_{h}).$
Let
 \[t_p:  A\rightarrow A\otimes \mathbb Z_p,\ a\rightarrow a\otimes 1, \]
 be the standard map, and obtain an algebra isomorphism
\[  H^\ast\left(t_p(\,\overline{V}),\bar d_h\right)\approx  H^*(\Omega
X;\mathbb Z_p).   \]
Furthermore, there is the coproduct
 \[\Delta: \overline{V}\rightarrow\overline{V}\otimes \overline{V}\]
 determined  for  $\bar x\in \overline{V}$ by the quadratic components of $d_h(x)$  in $RH,$ so that $ (\overline{V},\bar d_h)$ becomes a dg Hopf algebra.
 Consequently,
 the calculation of the Hopf algebra  structure of
$H^*(\Omega X;\mathbb{Z}_p)$ reduces to that of
$H^*(\,t_p(\overline{V}),\bar d_{h}).$

Note that the loop suspension homomorphism $\sigma$ is induced by
the composition map
\begin{equation}\label{sigma}
 RH \xrightarrow{pr} V\xrightarrow{s^{-1}}  \overline{V}, \ \ a\rightarrow
 \overline{pr(a) } ,
\end{equation}
and we immediately deduce the equality
\begin{equation}\label{susp}
 (\sigma x)^{p}=\sigma \mathcal{P}_1(x),\ \ \   x\in
\mathcal{S}^*(X)  ,
\end{equation}
and
the inclusion $\mathcal{D}^*_p(X)\subset \operatorname{Ker} \sigma$ as
well.

\subsection{The $(g,f)$-coderivation chain homotopy $\mathbf{b}$}\label{derhomotopy}

Given a cdga $H^\ast$ over $\mathbb{Z},$
let $\rho: (RH^\ast,d)\rightarrow H^\ast $ be a  multiplicative
resolution.   Let
\[
g: H^* \hookrightarrow R^*H^*\,\, \text{with}\,\,    g(x)\in R^0H^*\]
be a cocycle-choosing map; in general, this map is not additive
 homomorphism, but we only need its restriction to multiplicative
 generators $x\in \mathcal{S}^*_n,$ and can fix for $1\leq q\leq n$ 
 \[ g(x^q)=  g(x)^q . \]

Denoting $A:= (R^\ast H^\ast, d),$ let the bar-costruction
 $BA$ be  endowed with the    $\mu_{e}$
   product, while $BH$ with
   the shuffle $\ast:=sh_{\!_H}$ product.
   Since $R^* H^*$ is acyclic in negative
   resolution degrees (cf. subsection \ref{minimal}), the map $g$ extends to
   a comultiplicative map 
   \begin{equation}\label{fchain}
     \bar f=\{\bar f_k\}: BH\rightarrow  BA\ \  \text{with}\  \ f_k: B_kH \xrightarrow{\bar f_k}  \bar{A}\xrightarrow{s} A ,\  \ k\geq 1\ \ \text{and}\ \ \bar f_1=\bar g
       \end{equation}
       such that $\bar f$ is compatible with the differentials.
 Furthermore, $\bar f$ is  a
   homotopy multiplicative map, i.e., denoting
   \[
   g_1:= \bar f \circ sh\!_{_H}: BH\otimes BH\rightarrow BH   \rightarrow
   BA \]  and  \[g_2:= \mu _{e} \circ (\bar f\otimes \bar f ):  BH\otimes BH
   \rightarrow BA\otimes BA\rightarrow BA, \]
    there is
   a chain homotopy $\mathbf{b}:BH\otimes BH\rightarrow BA$  with
 \begin{equation}\label{bchain}
  g_2- g_1=d_{_{BA}} \mathbf{b}+\mathbf{b} \,d_{_{BH\otimes BH}}.
\end{equation}
Using again cofreeness
of $(BA,\Delta)$  choose $\mathbf{b}$ to be $(g_2,g_1)$ -- coderivation,
i.e., the diagram commutes
\begin{equation}\label{bcomult}
\begin{array}{ccc}
( BH\otimes BH)\otimes ( BH\otimes BH) &\xrightarrow{\mathbf{b}\otimes
g_1+g_2\otimes \mathbf{b}}  &  BA\otimes BA\\
    \hspace{-0.3in}_{\Delta \otimes \Delta} \uparrow & &
      \hspace{0.13in} \uparrow _{\Delta}      \\
 BH\otimes BH   & \xrightarrow{\mathbf{b}} & BA.
\end{array}
\end{equation}
In fact we need the values of $f$ and $\mathbf{b}$ on the following
special elements
of $BH$  and $BH\otimes BH.$ Namely, for a multiplicative generator
$x\in H$ with $\lambda x^{n+1}=0,\, n,\lambda \geq 1$ (e.g.,  $n=1$   and
$\lambda=2$ for an odd dimensional $x\in H^{od}(X)$),
 denote
\[
\begin{array}{ll}
u^1_k(x):=\begin{cases}

          [ \bar
         x\,|\,\overline{x^n}\,|\ldots|\,\bar{x}\,|\,\overline{x^n}\,|\,\bar
         x\,]\in B_kH, &  k \ \  \text{is odd}, \\
        [ \,\bar x\,|\,\overline{x^n}\,|\ldots|\,\bar x
        \,|\,\overline{x^n}\,]\in B_kH , & k\ \   \text{is even},
       \end{cases}
      \vspace{1mm} \\
            u^2_k(x):=\begin{cases}

          [
          \,\overline{x^n}\,|\,\bar
          x\,|\ldots|\,\overline{x^n}\,|\,\bar{x}\,|\,\overline{x^n}\,]\in
          B_kH, &  k\ \  \text{is odd}, \\
[ \,\overline{x^n}\,|\,\bar
        x\,|\ldots|\,\overline{x^n}\,|\,\bar x\,]\in B_kH , & k\ \
        \text{is even},
       \end{cases}
       \end{array}
       \]
and
\[x_{k-1}:=
\begin{cases}\,
          f_k(u^1_k(x))           & k\geq 1, \\
          \, f_k(u^2_k(x)), & k\ \   \text{is even},
        \end{cases}
           \]
           \[
           \!\!\! {x'}_{k-1}:= \lambda f_{k}(u^2_{k}(x)), \ \
           k\ \  \text{is odd}; \]
           precisely, 
           \begin{equation}\label{xprime}
 x'_{k-1}=\sum _{\substack{i_1+\cdots +i_n=s\\ i_j,s\geq 0}} \lambda\,
 x_{2i_1}\cdots x_{2i_n}\ \text{for}\  k=2s+1.
 \end{equation}
  In particular, $f_1([\,\bar x\,])= x_0$    and
           $f_1([\,\overline{x^n}\,])=x_0^n.$
   In general,   denote
   \[X^i_{k-1}:=
   f_k(u^i_k(x))\ \  \text{for all}\ \ k\geq 1\ \
   \text{and}\ \    i=1,2,  \]
  and
  \[
  {X^i_{k-1}\star X^j_{\ell-1}}:=\lambda
  f_{k+\ell}(u^i_k(x)\ast u^j_\ell(x))   \]
   where
      $u^i_k(x)\ast u^j_\ell(x)\in B_{k+\ell}H $ is the shuffle
      product.
      Then
      \begin{equation}\label{star}
      X^i_{k-1}\star X^j_{\ell-1}\!=\!\left\{
      \begin{array}{llll}
   \alpha_{k,\ell}(x)\,x_{k+\ell-1}, & \!x\in P^*_1, \vspace{1mm}\\
   \alpha^{i,j}_{k,\ell}(x)\,x_{k+\ell-1}+2v^{i,j}_{k,\ell}(x), &\! x\in
   P^*_{n>1}, (i,j)=(1,1),(2,2)\\
   & \text{or}\  k+\ell\ \text{is
   even},\vspace{1mm} \\
    \alpha^{i,j}_{k,\ell}(x)\,x'_{k+\ell-1}+ 2v_{k,\ell}^{i,j}(x), &
    \text{otherwise},
     \end{array}
     \right.
      \end{equation}
      where for $n=1:$
\[
  \alpha_{k,\ell}(x)=\left\{\begin{array}{llll}
   \binom{k+\ell}{k}, &    |x|\ \ \text{is odd}, \vspace{1mm} \\

      \binom{k+\ell-1/2}{k/2}, &     |x|\ \ \text{is even},\,\,
       k\  \text{is even},\,  \ell\  \text{is odd},
                               \vspace{1mm}\\
     \binom{k+\ell/2}{k/2},&
      |x|\ \ \text{is even},\
      k,\ell \ \text{are even},
               \vspace{1mm}\\
      0, &  |x|\ \ \text{is even},\ k,\ell\
                             \text{are odd},
                           \end{array}
                           \right.
  \]
  and for $n>1:$
 \[
  \alpha^{1,1}_{k,\ell}(x)=\alpha^{2,2}_{k,\ell}(x)=
  \left\{\begin{array}{llll}

                              \binom{k+\ell-1/2}{\ell/2}, &
                             k\  \text{is odd},\,  \ell\  \text{is even},
                                                           \vspace{1mm}\\
               \binom{k+\ell/2}{k/2},& k,\ell \ \text{are even},
               \vspace{1mm}\\
                             0, & k,\ell\  \text{are odd}, \\
                           \end{array}
                           \right.
  \]

     \[
  \alpha^{1,2}_{k,\ell}(x)=
  \left\{\begin{array}{llll}

     2,& |x|\  \text{is odd},\ \  k=1 ,\\

                              0,& |x|\  \text{is even},\ k=1 ,\vspace{1mm}\\

               \alpha^{1,1}_{k-1,\ell-1}(x)+1,& k,\ell>1 \ \text{are odd},
               \vspace{1mm}\\
                   \alpha^{1,1}_{k-1,\ell}(x)         , & k>1\ \text{is
                     odd},
                                          \ell\  \text{is even}. \\
                           \end{array}
                           \right.
  \]
Remark that we do not need an explicit form of the component
 $v^{i,j}_{k,\ell}(x)$   in (\ref{star}), because it is zero for $x\in P^{ev}_n$
 and only    appears   when $p=2$  and $x\in P^{od}_{n>1}.$
Finally, set
\begin{equation}\label{bkl}
\begin{array}{lll}
  b_{k,\ell}(x):= -\mathbf{b}(u_k(x)\otimes u_\ell(x)) , & n=1 ,
  \vspace{1mm}\\
 b^{i,j}_{k,\ell}(x):=   -\mathbf{b}(u^i_k(x)\otimes u^j_\ell(x)), &
 n>1,\,\, i,j\in\{1,2\}.
                       \end{array}
 \end{equation}

The  componentwise analysis of the
chain homotopy $\mathbf{b}$ just detects     the essential relations
among the elements $x_k$ and $b^{i,j}_{k,\ell}(x)$
 given by formula  (\ref{b-kl}) below.

\subsection{A minimal Hirsch resolution $RH$}
\label{minimal}

For a cga $H^\ast$
 with only relations  of the form  $x^{n+1}=0,\, n\geq 1,$
the  Hirsch resolution $R H$ can be described immediately.
Note that the essential idea can be seen for $n=1$ (the case $n>1$ is
somewhat technically difficult only).

Recall the construction of the Hirsch resolution $\rho: RH\rightarrow H.$
As an algebra $RH$ is a free bigraded
$\mathbb{Z}$-algebra
$R^\ast H^\ast=T(V^{*,*}),$
$V= \underset{i,j\geq 0}\bigoplus V^{-i,j},$
i.e.,
the multiplication respects the bidegree, the resolution differential
$d: RH\rightarrow RH$ is of the form
 $d:R^{-i}H^j\rightarrow R^{-i+1}H^{j}$  with  $d|_{R^0H^\ast}=0,$ and
 $H^{-i,\ast}(RH)=0$ for $i\geq 1.$ Let
$V^{*,*}=\langle \mathcal{V}^{*,*} \rangle ,$ so that the bigraded set
 \[\mathcal{V}^{*,*}= \underset{i,j\geq 0}\bigcup\mathcal{V}^{-i,j}\]
is multiplicative generators of $RH$
with $d(\mathcal{V}^{0,\ast})=0$ and  $\rho(\mathcal{V}^{i<0,\ast})=0.$

\subsection{The Hirsch algebra structure on $RH$}
The Hirsch algebra structure on $RH$ is in fact determined by the
 $\smile _{1}$\,-- product.
It
is defined by the equality
\begin{equation}\label{cup-one}
d(a\smile_1b )=
da\smile_1 b-(-1)^{|a|}a\smile _1db+(-1)^{|a|}ab-(-1)^{|a|(|b|+1)}ba;
\end{equation}
it is associative on $V,$ and  extended on the decomposables
by the following two equalities:
\begin{enumerate}
\item
\textit{The (left) Hirsch formula.} For $a,b,c\in RH:$
\begin{equation}\label{hirsch1}
c\smile _{1}ab=(c\smile _{1}a)b+(-1)^{(|c|+1)|a|}a(c\smile _{1}b)
\end{equation}
\item
 \textit{The (right) generalized Hirsch formula.} For $a,b\in
RH$ and $c\in V$ with $d_{h}(c)=\sum c_{1}\cdots
c_{q}$   for  $c_{i}\in V:$
\begin{equation} \label{hirsch2}
ab \smile _{1} \!c
 = \left\{ \!
\begin{array}{llll}
a(b\smile _{1}\!c)+(-1)^{|b|(|c|+1)}(a\smile _{1}\!c)\, b, &
q=1,\vspace{5mm}
 \\
a(b\smile _{1}\!c)+(-1)^{|b|(|c|+1)}(a\smile _{1}\!c)\,b\,+
\vspace{1mm}
\\
\underset{1\leq i<j\leq q}{\sum }\! (-1)^{\varepsilon}\,c_{1}\cdots
c_{i-1}(a\smile _{1}\!c_{i})\,c_{i+1}\vspace{1mm}  \\
\hspace{0.6in}\cdots c_{j-1}(b\smile _{1}c_{j})\, c_{j+1}\cdots
c_{q}, & q\geq 2,
\end{array}
\right.
\end{equation}
where  $\varepsilon=(|a|+1)\left(|c_1|+\cdots +|c_{i-1}| \right)+
(|b|+1)\left(|c_1|+\cdots +|c_{j-1}|\right);$
\end{enumerate}
in particular, for $dc=c_1c_2,$
\begin{multline}\label{quadratic}
ab\smile_1 c=\\
a(b\smile _{1}\!c)+(-1)^{|b|(|c|+1)}(a\smile _{1}\!c) b \,+
(-1)^{(|b|+1)|c_1|}(a\smile_1c_1)(b\smile_1c_2)     .
\end{multline}

We also have another binary product  on $RH,$ denoted by
 $\sqcup_2,$  thought of as a quasi Steenrod's cochain $\smile_2$\,--
 product.
It
is defined  by  $a\sqcup_2 a=0$ for an odd dimensional $a$
and 
\begin{equation}\label{uplus}
d(a\sqcup_2 b )=
\left\{
\begin{array}{lllll}
da\sqcup_2 a+a\smile_1 a,  &  \,\, a=b, \,\,\,\,  |a|\,\,   \text{is
even};\vspace{1mm}\\
2 b\sqcup_2 b-a\smile_1 b+b\smile_1 a,  &  da=b,\,\,\, |b| \,\, \text{is
even};\vspace{1mm}\\
da\sqcup_2  b+(-1)^{|a|}a\sqcup_2 db+  \\
 (-1)^{|a|}a\smile_1 b+(-1)^{(|a|+1|)|b|}b\smile_1 a, &
 \text{otherwise}.
\end{array}
\right.
\end{equation}
 Denoting
 \[
 \epsilon\cdot (a\sqcup_2 b)=\begin{cases}
 2 (a\sqcup_2 b),& a=b,\\
  a\sqcup_2 b, & \text{otherwise}
 \end{cases}
 \]
 extend $\sqcup_2$ on the decomposables
by the following two equalities:
\[
\begin{array}{ll}
c\sqcup_2 ab=
(\epsilon\cdot c\sqcup_2 a)b+(-1)^{|c||a|}a(\epsilon\cdot c\sqcup_2 b)
\ \ \text{and} \vspace{1mm}\\
ab \sqcup_2 \!c
 = \left\{ \!
\begin{array}{llll}
a(\epsilon\cdot b\sqcup_2\!c)+(-1)^{|b||c|}(\epsilon\cdot
a\sqcup_2\!c)\, b, &  dc\in V, \vspace{2mm}
 \\
a(\epsilon\cdot b\sqcup_2\!c)+(-1)^{|b||c|}(\epsilon\cdot
a\sqcup_2\!c)\,b\,+
\vspace{1mm}
\\
 (-1)^{(|b|+1)|c|} (\epsilon\cdot a\sqcup_2 c_1)
 ( c_2\smile_1b )\,-\vspace{1mm}\\
 (-1)^{|b|(|c_1|+1)}
 (a\smile_1 c_1)(\epsilon\cdot b\sqcup_2 c_2), & dc=c_1c_2;
\end{array}
\right.
\end{array}
\]
it is  commutative on $V,$
$a\sqcup_2 b=(-1)^{|a||b|}b\sqcup_2 a,$
and is associative unless on a string  of even dimensional elements $a$
of length $2^k$
\[ U_{2^k}:=a\sqcup_2 \cdots \sqcup_2 a; \]
 namely, if $U_{2^k}$ is defined for $k\geq 1,$ then
 $U_{2^{k+1}}=U_{2^k}\sqcup_2 U_{2^k}.$
The relations between $\smile_1$   and  $\sqcup_2$ are:
\begin{equation*}
c\sqcup_2(a\smile_1b)=(\epsilon \cdot c\sqcup_2
a)\smile_1b+(-1)^{|c||a|}a\smile_1(\epsilon \cdot c\sqcup_2 b),
\end{equation*}
\begin{equation*}
(a\smile_1b)\sqcup_2 c =a\smile _{1} (\epsilon \cdot b\sqcup_2
c)+(-1)^{|b||c|}(\epsilon \cdot a\sqcup_2 c)\smile _1 b,
\end{equation*}
\begin{multline*}
 (ab)\sqcup_2(xy)= a(\epsilon\cdot b\sqcup_2 x)y
 +(-1)^{|b||x|}(\epsilon\cdot a\sqcup_2 x)by\, +
 \\
 (-1)^{|x||ab|}xa(\epsilon\cdot b\sqcup_2
 y)+(-1)^{|x||ab|+|b||y|}x(\epsilon\cdot a\sqcup_2 y)b\, -
\\
(-1)^{|a|+|ab||x|+|y|(|b|+1)}(x\smile_1a)(y\smile_1 b).
\end{multline*}

\subsection{The description of multiplicative generators of $RH$}

Below we describe the certain subsets of $\mathcal{V}^{\ast,\ast}$
needed in the sequel. Fix a prime $p,$ and let
 $\mathcal{H}^*_p\subset H^*(X;\mathbb Z_p)$  be
a set of multiplicative generators. Define the subset in $\mathcal{V}^{-1,*},$
\[
\aleph^{-1,*}_{p^r}=\left\{c\in \mathcal{V}^{-1,*}\mid dc= p^r  c_0,\,\,  c_0\in \mathcal{V}^{0,*}\right\}\ \ \text{and}\ \
\mathcal{O}^{-1,\ast}_{p}=\bigcup_{r\geq
1}\mathcal{\aleph}^{-1,\ast}_{p^r}.
\]
Let  $ x=[t_p(x_0)]\in \mathcal{H}_p$  for $x_0\in
\mathcal{V}^{0,\ast}\cup \mathcal{O}^{-1,\ast}_p.$ Then
$\mathcal{H}^\ast_p=\mathcal{H}^\ast_{0,p} \cup
  \mathcal{H}^\ast_{1,p}$
 with
 $\mathcal{H}^\ast_{0,p}= \{
 x\in  \mathcal{H}^\ast_p\mid  x_0\in  \mathcal{V}^{0,\ast} \}$ and
 $\mathcal{H}^\ast_{1,p}:= \{
 x\in \mathcal{H}^\ast_p\mid  x_0\in  \mathcal{O}^{-1,\ast}_p \};$
 in particular,
  $\beta(x)=0$  for $x\in \mathcal{H}^\ast_{0,p},$ and  $\beta_r(x)\neq
  0$  for $x\in \mathcal{H}^\ast_{1,p}.$  Furthermore, when $x\in P^*_n\cap
  \mathcal{H}_p  $
  with $\beta(x)\neq 0,$
  there are
   $x_1\in \mathcal{V}^{<0,\ast}$  and  $\tilde{x}_1\in \mathcal{V}$ such that for  $d x_0=p \tilde x_0$
   \begin{equation}\label{x1}
     dx_1=
   (-1)^{|x_0|+1}\lambda x_0^{n+1}+ p\tilde{x}_1\ \  \text{and}\ \
   d\tilde x_1= \underset{i+j=n}{\sum} \lambda\, x^i_0\tilde x_0
      x_0^{j}
      \end{equation}
   with $\lambda$ to be    not divisible by $p.$
      Define the subset $\mathcal{X}^{-1,\ast}_p\subset \mathcal{V}^{-1,
 \ast}$ as
  \[\mathcal{X}^{-1,*}_p  =\left\{ x_1\in \mathcal{V}^{-1,*}\mid dx_1=
  (-1)^{|x_0|+1}\lambda x_0^{n+1}, \, \, x_0\in \mathcal{V}^{0,\ast}\right \}
  .\]
Let
$\mathcal{E}^{-1,\ast}=\{a\smile_1 b\in \mathcal{V}^{-1,* }\mid  a,b \in
\mathcal{V}^{0,\ast} \}.$
Then fix the subset $\mathcal{V}^{-1,\ast}_p\subset
\mathcal{V}^{-1,\ast}$  by
 \[
\mathcal{V}^{-1,\ast}_p=\mathcal{O}^{-1,\ast}_p\cup\mathcal{X}^{-1,\ast}_p
\cup\mathcal{E}^{-1,\ast}.\]
To describe  certain elements in $\mathcal{V}^{-i,\ast}_p$ for $i\geq 2,$ first remark about
the set  $\mathcal{V}^{-2,\ast}.$ For $x_0\smile_1 x_0\in
\mathcal{E}^{-1,\ast}$ with an even dimensional $x_0\in
\mathcal{V}^{0,ev}$ we have that $d(x_0\smile_1 x_0)=0,$ and to
achieve the acyclicity in $R^{-1}H^*,$ the generator $x_0\smile_1 x_0$
must be killed
by some generator   from  $\mathcal{V}^{-2,\ast};$ namely,
(\ref{uplus}) implies that such a generator is just
$x_0\sqcup_2 x_0\in \mathcal{V}^{-2,*}$ with
\[ d (x_0\sqcup_2 x_0)= x_0\smile_1 x_0 . \]
Furthermore,  relation (\ref{x1}) gives rises to the infinite
sequence in $RH$
\[\{x_k\}_{k\geq 0}\]
that subject to the
following  relations being in correspondence with  the fact that
 the map $\bar f$    given by (\ref{fchain}) is compatible with the
 bar-construction  differentials
\begin{equation}\label{syzygies1}
dx_{2k+1}=
  - \underset{\substack{i+j=k\\i,j\geq 0}}{\sum}x_{2i}x'_{2j}+
 \underset{\substack{s+t=k-1\\s,t\geq 0}}{\sum} x_{2s+1}x_{2t+1}+ p\,
 \tilde{x}_{2k+1},
 \end{equation} 
 \begin{equation}\label{syzygies2}
\hspace{-0.65in}dx_{2k}= \underset{\substack{i+j=2k-1\\i,j\geq 0}}{\sum}
(-1)^{|x_{i}|+1}x_{i}x_{j}+ p\,\tilde{x}_{2k},
\end{equation}
where $x'_{2s}$ is given by (\ref{xprime}).
 In particular,   for $n,\lambda =1$
  \begin{equation}\label{oddsyzygies}
   dx_{k}=
  \underset{\substack{i+j=k-1\\i,j\geq 0}}{\sum}  (-1)^{|x_{i}|+1}
  x_{i}x_{j}
  +p\,\tilde{x}_{k}\ \  \ \   \text{with}\ \
     \end{equation}
     \[\hspace{-0.5in} d\tilde x_{k}=
  \underset{\substack{i+j=k-1\\i,j\geq 0}}{\sum}  (-1)^{|x_{i}|} \tilde
  x_{i}x_{j}+x_{i}\tilde x_{j}.\]
 Define the set $\mathcal{X}_p$ by $\mathcal{X}_p=\mathcal{X}_{p,1}\cup
 \mathcal{X}_{p,2}$ for
 \[
 \begin{array}{ll}
 \mathcal{X}_{p,1}=
 \left\{ \{ x_0,x_{p-1},...,x_{p^{r}-1},...\,\}_{r\geq 0},
 \hspace{0.55in}
 x\in \mathcal{H}^\ast_p\cap P^{*}_1  \right\},
   \vspace{1mm} \\
 \mathcal{X}_{p,2}=  \left\{  x_0\cup \{
 x_{1},x_{2p-1}...,x_{2p^{r-1}-1},...\,\}_{r\geq 1},\ \
  x\in  \mathcal{H}^\ast_p\cap  P^\ast_{n>1} \right\}
  .
\end{array}
\]
Furthermore, the elements $b_{k,\ell}^{\ast,\ast}(x)$ given by
(\ref{bkl})
imply the following sequence of relations
\begin{multline}\label{b-kl}
d{b}^{1,1}_{k,\ell}(x)= \alpha^{1,1}_{k,\ell}(x)\,x_{k+\ell-1} +
x_{k-1}\!\smile_1\! x_{\ell-1}-\\
\hspace{1in}\sum_{\substack{0\leq r<k\\ 0\leq m<\ell }}\!
\left(
(-1)^{\epsilon_1} b^{1,1}_{k-r,\ell-m}(x)\cdot X^i_{r-1}\star X^j_{m-1}
+\right.
\\
\left.
(-1)^{\epsilon_2}\!
\left( x_{r-1}\!\smile_1\! x_{m-1}\right)
b^{\ast,\ast}_{k-r,\ell-m}(x)\right)
+
p\, \tilde{b}^{1,1}_{k,\ell }(x),
\end{multline}
with the convention
 ${x}_{-1}\smile_1 {x}_{t}={x}_{t}\smile_1{x}_{-1} =-x_t,$
 and
 \[
\begin{array}{lll}
\epsilon_1=(\,|{X}_{r-1}|+1)(\,|{X}_{m-1}|
+|{X}_{\ell-1}|) +|{X}_{k-1}|+|{X}_{\ell-1}|+
|{X}_{r-1}|+|{X}_{m-1}|,\vspace{1mm}\\
\epsilon_2=(\,|{x}_{m-1}|+1)(\,|{x}_{k-1}|
+|{x} _{r-1}|).
\end{array}
\]
Note that the upper indices of $b^{*,*}_{k-r,\ell-m}$ are uniquely
determined
by the parity of integers $k,r$ and $\ell,m.$
Furthermore,  $b_{11}^{11}(x)=0$ for $x\in P^{od}_{n>1}$ while
for even dimensional $x\in P^{ev}_n$ we can set
$b_{11}^{11}(x)=x_0\sqcup_2 x_0.$
In particular,
 for $x\in P^{od}_1$ formula (\ref{b-kl}) reads:
\begin{multline}\label{bodd}
d{b}_{k,\ell }(x)=\binom{k+\ell}{k}{x}_{k+\ell -1}+
{x}_{k-1}\smile_1 {x}_{\ell-1}- \\
\hspace{0.9in}\sum_{\substack{ 0< r<k \\0< m<\ell}}\!\!
\left(\!\!
\binom{r+m}{r}
{b}_{k-r,\ell-m}(x)\, x_{r+m-1}+
(x_{r-1}\!\smile_1x_{m-1})\,{b}_{k-r,\ell-m}(x)
\right. \\
\hspace{1.0in}
\left.
 +\,{b}_{k-r,\ell}(x)\, x_{r-1} - x_{r-1} b_{k-r,\ell}(x) \right.\\
 \left.
 \hspace{2in} +\, b _{k,\ell-m}(x)\,x_{m-1} - x_{m-1}b _{k,\ell-m}(x)
\right)
+
p\, \tilde{b}_{k,\ell }(x).
\end{multline}
The first equalities of (\ref{b-kl}) for $x\in P^{od}_{n}$
with $\tilde{b}^{*,*}_{k,\ell}(x)=0$ (in particular, $\beta(x)=0$)
read:
\[
\begin{array}{lll}
\\
db^{12}_{11}(x)=2x_1+x_0\smile_1x_0^n,\ \ \ n\geq 1,\vspace{1mm}\\

db^{11}_{12}(x)=x_2+2v_3(x)+x_0\smile_1x_1 +x_0\,b^{12}_{11}(x),\ \ \
n>1, \vspace{1mm}\\

db_{22}(x)=6x_3 +x_1\smile_1x_1 -  2b_{11}(x)\,
x_1-
\left(x_0\smile_1 x_0 \right) b_{11}(x) -\\
\hspace{1.7in}\left(b_{12}(x)+b_{21}(x)\right)x_0+ x_0
\left(b_{12}(x)+b_{21}(x)\right), \,\, n=1, \vspace{1mm}
\\

db^{11}_{22}(x)=2x_3+  2v_4(x) +x_1\smile_1x_1 -
\left(x_0\smile_1 x_0 \right) b^{22}_{11}(x)-\\
\hspace{1.7in}\left(b^{11}_{12}(x)+b^{11}_{21}(x)\right)x^n_0+ x_0
\left(b^{21}_{12}(x)+b^{12}_{21}(x)\right), \,\,  n>1.
\end{array}
\]
For $x\in P^{ev}_{n},\, n\geq1,$ the above equalities  change as:
\[
\begin{array}{lll}
\\
db^{12}_{11}(x)=x_0\smile_1x_0^n  \ \ \   (\text{in fact}\ \
b^{12}_{11}(x)=\underset{i+j=n-1}{\sum} x_0^i\,
b_{11}^{11}(x)\,x_0^j),\vspace{1mm}\\

db^{11}_{12}(x)=x_2+x_0\smile_1x_1 -b^{11}_{11}(x)\,x_0^n -
x_0\,b^{12}_{11}(x),\vspace{1mm}\\

db^{11}_{22}(x)=2x_3 +x_1\smile_1x_1 +(x_0\smile_1 x_0 )\,
b^{22}_{11}(x) +  \\
\hspace{2in}\left(b^{11}_{12}(x)+b^{11}_{21}(x)\right)x^n_0+ x_0
\left(b^{21}_{12}(x)+b^{12}_{21}(x)\right).
\end{array}
\]

\vspace{5mm}
Define the set
$\mathcal{B}_p=\mathcal{B}_{p,1}\cup\mathcal{B}_{p,2}$
for
\[
\begin{array}{lll}
\mathcal{B}_{p,1}=  \left\{ b_{k,\ell}(x)     \mid
  (k,\ell)\in((1, p^r-1), ( p^{r}, (p-1)p^r)), \ \ r\geq1, \,\, x\in
  \mathcal{H}^*_p\cap  P^{*}_{1}  \,\right\},   \vspace{1mm} \\
  \mathcal{B}_{p,2}=  \left\{ b_{k,\ell}(x)     \mid
  (k,\ell)\in(( 2, 2(p^r-1)),(2 p^{r-1},2(p-1)p^{r-1}) ),\ \   r\geq 1,
  \right.\\
  \left.
  \hspace{3.6in} x\in  \mathcal{H}^*_p\cap P^*_{n>1}
   \right\}.
\end{array}
 \]
 \begin{remark}\label{monom}
 Let  \begin{equation}\label{decompos}
        \mathcal{D}^*_p(RH)=p\cdot \!RH + RH^+\!\cdot RH^+  \subset RH.   
      \end{equation}
 Note that  checking the divisibility of $\alpha_{k,\ell}(x)$ by $p$ in (\ref{b-kl}), if $m=k+\ell-1$ and
   $x_m\notin  \mathcal{X}_p,$ then $x_m$ can be  defined  as 
   $x_m=x_{k-1}\smile_1 x_{\ell-1} \mod \mathcal{D}^*_p(RH)$
    in which case we can set $b_{k,\ell}(x)=0\mod p.$ More precisely, by successive application of the argument we can achieve
   \begin{equation}\label{cup1monom}
     x_m=x_{i_1}\smile_1\cdots \smile_1 x_{i_s} \mod \mathcal{D}^*_p(RH) 
     \end{equation}
   for $0\leq i_1<\cdots i_s<m$    and $x_{i_j}\in \mathcal{X}_p $ for all $j.$
   \end{remark}
 
Finally,  define the subset $\mathcal{U}\subset \mathcal{V}$ as
$\mathcal{U}=\left \{ a_1\sqcup_2\cdots \sqcup_2 a_k\mid a_i\in
\mathcal{V} ,\, k\geq 1 \right\}. $
Thus,  we obtain
\[
 \mathcal{V}_p: = \mathcal{X}_p\cup \mathcal{B}_p\cup
 \mathcal{U}\subset \mathcal{V}.
 \]

\subsection{The perturbation $h:RH\rightarrow RH$}
A perturbation $h$
of the differential $d$  is an additive  homomorphism $h:RH\rightarrow
RH$ and a derivation as well
\begin{equation}\label{aderivation}
h(a\cdot b)=ha\cdot b+(-1)^{|a|}a\cdot hb
\end{equation}
such that
\[d^2_h:=(d+h)^2   =  dh+hd + hh=0, \]
and is determined up to the action of the automorphism group on $RH$
 (cf. Theorem 1  in \cite{saneFiltered}). The perturbation
   $h$ is canonically graded as
  \[h=h^2+h^3+\cdots +h^r+\cdots \ \  \text{with}\ \
  h^r:R^{-s}H^{t}\rightarrow  R^{-s+r}H^{t-r+1},\]
  and $r$ is called \emph{the perturbation degree of $h$}.
  The equality $d_h^2=0$ yields a sequence of equalities with respect
   the perturbation degree $r$ the first of which is $dh^2+h^2d=0.$
  In particular,
    \[h|_{R^{-1}H\oplus R^0H}=0.\]
    Similarly to the set $\aleph^{-1,*}_p,$
   define the subset $\aleph^{-1,*}\subset R^{-1}H^*$  as
 \[\aleph^{-1,*}:= \{ c\in R^{-1}H^* \mid dc=\alpha\cdot
 c_0,\,\,\alpha\in \mathbb Z,\,  c_0\in R^0 H^*   \}. \]
  The \emph{transgressive} term $h^{tr}(a)$ of $ha$ is defined by the
  restriction of $ha$ to $\aleph^{-1,*} + R^0H^*$
\[ h^{tr}(a)=   ha|_{\aleph^{-1,*}+ R^0H^*} \ \ \text{for} \ \
a\in R^{\leq -2}H^*. \]
For  $a\in RH,$ let $h^{\prime}(a)$ denote $h(a)$ without the transgressive
terms:
 \[h^{\prime}(a)=h(a)- h^{tr}(a).\]
 Obviously, when $p$ divides $\alpha,$   we have the equality of
 cohomology classes
\begin{equation}\label{split}
  [  d_{h^{\prime}}(a)  ] = - [h^{tr}(a)   ]   \mod p.
\end{equation}
Furthermore,
the perturbation $h$  is  a $\smile_1$\,--\,derivation, too,
\begin{equation}\label{1derivation}
 h(a\smile_1b)=ha\smile_1 b-(-1)^{|a|}a\smile_1 hb,\ \ \ \ \  a,b\in
 RH.
 \end{equation}
 In particular,
$h^{tr}(a\smile_1 b)=0$ for all $a,b\in RH.$

\begin{remark}
  Unlike cup-- and cup--one products the perturbation $h$ is not
  derivation with respect to the
  $\sqcup_2$\,-- product, and unlike the    $\smile_2$-- product the
  $\sqcup_2$-- product does not exist in $C^\ast(X;\mathbb Z)$
  canonically.
\end{remark}

\subsection{The perturbation $h$ on
$\mathcal{V}_p$}\label{perturbation}
The perturbation $h$ on $\mathcal{V}_p$  is purely determined by its
transgressive terms $h^{tr}$ via explicit formulas.
Given $k\geq 0,$ denote $ h^{tr}(x_k)$ by
\[ y_{k}=y^{-1}_{k}+y^{0}_{k}  \ \ \text{with}\ \
y^{-1}_k\in {\aleph}^{-1,*}\ \  \text{and} \ \ y^0_k\in R^0H^*,
\]
  where $y_0=y_1=0$  for any $x\in P^*_n.$
Given $m\geq 1,$ let $P_\ast(m)$ be  a set of  sequences
\[P_q(m)=\{  \mathbf{i}=(i_1,...,i_q)\mid m= |\mathbf{i}|:= i_1+\cdots +i_q,\, i_j\geq 1  \},
\]
and let $\overline{P}_q(m): =  P_q(m-q+1)$ for $q\geq1 .$
Define
 \[
   Y_m=\sum_{\mathbf{i}\in \overline{P}_\ast(m)} Y^\mathbf{i}_m,\,\, \,\,\,\,
   Y^\mathbf{i}_m=\begin{cases}
                     y_{i_1}\sqcup_2\cdots \sqcup_2 y_{i_q}, & q\geq 2,
                     \\
                    y_{m}, & q=1
                  \end{cases}
 \]
 with
 \begin{equation}\label{Y}
  d Y^\mathbf{i}_m=\sum_{\mathbf{i}={\mathbf{i}_1}\cup {\mathbf{i}_2}} Y^{\mathbf{i}_1}_{m_1} \smile_1 Y^{\mathbf{i}_2}_{m_2}\ \ \text{for}\ \ 
 q\geq2.
 \end{equation}
 In particular, $Y^\mathbf{i}_m=Y^\mathbf{i'}_m$
 when $\mathbf{i}$  and $\mathbf{i}'$ differ from each other by a permutation  of components.
Let $n+1=2r,\, r\geq 1.$ Given an odd dimensional  $x\in P^{od}_n$ and
$m\geq 0,$  form  the sum of monomials obtained from
\begin{equation}\label{sum}
\sum _{\substack{i_1+\cdots +i_{2r}=m\\ i_j\geq 0}} \lambda\,
x_{2i_1}\cdots x_{2i_{2r}}
  \end{equation}
 by all possible  replacements
 \[        x_{2i_{2s-1}}x_{2i_{2s}}  \leftarrow  \, - \,
 x_{2i_{2s-1}}\!\smile_1 Y_{2i_{2s}} \]
 and
 \[
  \hspace{1.0in} x_{2i_{2t-1}}x_{2i_{2t}}  \leftarrow  \,\,
  Y_{2i_{2t-1}}  \sqcup_2 Y_{2i_{2t}},\ \ \ \ \ \ \  1\leq s,t\leq r
   \]
   in which the sum of monomials of  the form
\[x_0^{n_1}  Y_{i_1}\cdots   x_0^{n_s}\cdots Y_{i_k}  x^{n_t}_0,\ \ \
n_1+\cdots +n_t=n-2k+1
 ,\,n_s\geq 0,\, 1\leq k\leq r,
\]
 is denoted by
$_n\!Y_{2m},$ and the sum of the other ones  by $Z_{2m}.$
When $n=1,$ obtain
\[Z_{2m}=\sum_{\substack{i+j=m\\ i\geqslant 0;j\geqslant 1}} -
x_{2i}\smile_1 Y_{2j} \ \  \text{with}\  \         _1Y_{2m}=Y_{2m}.\]
Define also $_n\!Y_{2m+1}=Y_{2m+1},\, n\geq1,$  and then  define $h$ on $\mathcal{V}_p$ as follows.
\medskip

\noindent On $\mathcal{X}_p:$ (i) Let $x$ be odd dimensional, $x\in
P^{od}_n.$
In particular, $x_k$ is also odd dimensional for all $k.$ Define
\begin{equation}\label{h-odd1}
h(x_{2m+1})=
Z_{2m}+\, _n\!Y_{2m}
-\sum_{\substack{i+j=m-1\\ i\geqslant 0;j\geqslant 1}}
x_{2i+1}\smile_1 Y_{2j+1}+ Y_{2m+1} +ph(\tilde{x}_{2m+1}),
\end{equation}
\begin{equation}\label{h-odd2}
h(x_{2m})=
-\sum_{\substack{i+j=2m-1\\ i\geqslant 0;j\geqslant 1}}
x_{i}\smile_1 Y_{j}+  Y_{2m}  + ph(\tilde{x}_{k}).
\end{equation}
\medskip
(ii) Let $x$ be even dimensional, $x\in P^{ev}_n.$ This time every $x_k$ is
no longer odd dimensional, but $k$ and $x_k$  have the same parity.
To control signs in $h(x_k)$
    for $a,b\in RH$   define the element $a\Cup_1b\in RH$ by
    \[
    \begin{array}{rllll}
d(a\Cup_1b )&=&
da\Cup_1 b-(-1)^{|a|}a\Cup _1db-ab- ba\ \  \text{and}\ \    b\Cup_1
b=0,  & |b| & \text{is odd},\vspace{1mm}\\

a\Cup_1b& =  & a\smile_1 b, & |b| & \text{is even},
\end{array}
 \]
 and for $da=a_1a_2$ with $a_1,a_2\in V$
 \[ (a_1a_2) \Cup_1 b= (-1)^{|a_1|+1}a_1(a_2\Cup_1 b)+(a_1 \Cup_1
 b)a_2, \ \ \ |b|\ \ \text{is odd} .   \]
  Let $Z^{ev}_{2m}$ be defined as $Z_{2m}$ but $\smile_1$ to be replaced
 by $\Cup_1.$
 When $n=1,$ obtain
\[Z^{ev}_{2m}=\sum_{\substack{i+j=m\\ i\geqslant 0;j\geqslant 1}} -
x_{2i}\Cup_1 Y_{2j}.\]
Set $h(a\Cup_1b )=0$ for $b$ to be odd dimensional, and 
then define
\begin{equation}\label{h-even1}
h(x_{2m+1})=
Z^{ev}_{2m}+ \,_n\!Y_{2m}
-\sum_{\substack{i+j=m-1\\ i\geqslant 0;j\geqslant 1}}
x_{2i+1}\smile_1 Y_{2j+1}+ \\
Y_{2m+1}
 +ph(\tilde{x}_{2m+1}),
\end{equation}
\begin{equation}\label{h-even2}
h(x_{2m})=
-\sum_{\substack{i+j=2m-1\\ i\geqslant 0;j\geqslant 1}}
x_{i}\smile_1 Y_{j}+  Y_{2m}  + ph(\tilde{x}_{k}).
\end{equation}
When $\tilde{x}_k\neq 0$ we can choose $h$
(expressed by $\smile_1$-- and $\sqcup_2$ -- products) with
$h^{tr}(\tilde{x}_k)=0$ for all $k.$ Hence,  $y^{-1}_k=0$  and
$y_k=y^0_k.$
 Furthermore,  the equality    $d^2_h(x_k)=0$ uses the fact that
  $h(Y_k)=0$ for  all $k.$ Indeed, consider  $x_k$ such  that
  $hx_k$ contains   
   $Y^{\mathbf{i}}_{k}$ with $|\mathbf{i}|=2.$  We have $d(Y^{\mathbf{i}}_{k})=(dh+hh)(x_k).$ On the other hand,
      the cocycle $ (\varphi \circ \rho^{-1})( (dh+hh)(x_k)  )$
is cohomologous to zero in $C^*(X;\mathbb Z),$  and,   consequently,
     $h^{tr}( Y^{\mathbf{i}}_{k})  =0$  in $RH.$ 
     In view of   $(\ref{Y})$  we can inductively on the length of 
     $\mathbf{i}$
       achieve 
          $h^{tr}( Y^{\mathbf{i}}_{k})  =0$ for $|\mathbf{i}|\geq 2.$
      \begin{remark}
      In general, $h(a\sqcup_2 a)$ may be not zero for an even
      dimensional $a$ (see the last paragraph of this section).
      However, the established
       equality   $h(y \sqcup_2 y)=0$  could be considered as a
      proof of the fact that $Sq_1$ annihilates  the symmetric Massey
      products in $H^*(X;\mathbb Z_2)$   (cf. \cite{saneFiltered}).
    \end{remark}
      In particular,  the first four  equalities of
      (\ref{h-odd1}) with
$\tilde{x}_k=0$ are:
\[\begin{array}{c}
       hx_2=y_2 ,\\
    hx_3=-x_0 \smile_1 y_2+  y_3, \\
     hx_4=-x_0 \smile_1 y_3- x_1 \smile_1 y_2 + y_4,  \\
     hx_5=-x_0 \smile_1 y_4- x_1 \smile_1 y_3 -
     x_2 \smile_1 y_2 + y_2\sqcup_2 y_2 + y_5.
  \end{array}
  \]

\medskip

   \noindent On $\mathcal{B}_p:$
      For $k,\ell\geq 1$   denote
   \begin{equation}\label{trans}
     c^{*,*}_{k,\ell}(x):= -
   h^{tr}(b^{*,*}_{k,\ell}(x))\ \
   \text{with}\ \  c_{k,\ell}=c^{-1}_{k,\ell}(x)+c^0_{k,\ell}(x)\in \aleph^{-1}H^*+R^0H^*.
   \end{equation}
   Given  a pair $(k,\ell)$ and $t\geq 1,$  for $\mathbf{s}=(s_1,...,s_t)\in  P_t(k+\ell-t)$ 
    define  the elements $C^{\mathbf{s}}_{k,\ell}(x)\in RH$ inductively
    with
    $C^\mathbf{s}_{k,\ell}=C^\mathbf{s'}_{k,\ell}$
                    when $\mathbf{s}$  and $\mathbf{s}'$ differ from each other by a permutation of components as follows.
   Let for $t=1$ and $\mathbf{s}=(s_1)=(k+\ell-1)$
    \begin{equation}\label{cc}
       C^{(k+\ell-1)}_{k,\ell}(x): =c_{k,\ell}(x),
    \end{equation}
    and for $t>1$ 
    denoting
   \begin{equation*}\label{ekl}
     \epsilon_{k,\ell}=\begin{cases}
    1, &   k\neq \ell  , \\
               2, &    k=\ell,
             \end{cases}
             \end{equation*}
let for $x\in P^{od}_1:$
     \begin{multline}\label{C-kl-od1}
   d C^\mathbf{s}_{k,\ell}(x)   =
   \binom{k+\ell}{k} Y^\mathbf{s}_{k+\ell-1}-\sum_{\mathbf{i}_k\,\cup\,\mathbf{j}_\ell=\mathbf{s}}
   \epsilon_{k,\ell} \cdot Y^{\mathbf{i}_k}_{k-1}\sqcup_2
   Y^{\mathbf{j}_\ell}_{\ell-1}
   +
   \\
   \hspace{0.5in}\sum_{\mathbf{i}\,\cup\,\mathbf{j}=\mathbf{s}}
      \sum_{\substack{0< r<k \\ 0< m<\ell}}
\left(\binom{r+m}{r}
Y_{r+m-1}^\mathbf{i}\smile_1 C^\mathbf{j}_{k-r,\ell-m}(x)+\right.\\
\left.
\hspace{2.28in}
Y_{r-1}^\mathbf{i}\smile_1
C^\mathbf{j}_{k-r,\ell}(x)+Y_{m-1}^\mathbf{i}\smile_1
C^\mathbf{j}_{k,\ell-m}(x)
-
\right.\\
\hspace{2.28in}
\left.
C_{k-r,\ell}^\mathbf{j}(x)\smile_1 Y^\mathbf{i}_{r-1}-
C^\mathbf{j}_{k,\ell-m}(x)\smile_1 Y^\mathbf{i}_{m-1}-\right.
\\
\left.
C^\mathbf{i}_{r,m}(x)\cdot C^\mathbf{j}_{k-r,\ell-m}(x)\,\,\right);
    \end{multline}
 \noindent for $x\in P^{od}_{n>1}$ (i.e., $p=2$):
       \begin{multline}\label{C-kl-odn}
   d C^\mathbf{s}_{k,\ell}(x)   =
   \alpha^{1,1}_{k,\ell}(x)\cdot {_n\!Y}^\mathbf{s}_{k+\ell-1}-
   \sum_{\mathbf{i}_k\,\cup\,\mathbf{j}_\ell=\mathbf{s}}
   \epsilon_{k,\ell} \cdot {_n\!Y}^{\mathbf{i}_k}_{k-1}\sqcup_2
   {_n\!Y}^{\mathbf{j}_\ell}_{\ell-1}
   +
   \\
   \hspace{0.5in}\sum_{\mathbf{i}\,\cup\,\mathbf{j}=\mathbf{s}}
      \sum_{\substack{0< r<k \\ 0< m<\ell}}
\left(
\left(\alpha^{*,*}_{r,m}(x)\,
{_n\!Y}_{r+m-1}^\mathbf{i}+2\,hv^{*,*}_{k,\ell}(x)\right)\smile_1
C^\mathbf{j}_{k-r,\ell-m}(x)+\right.\\
\left.
\hspace{2.1in}
{_n\!Y}_{r-1}^\mathbf{i}\smile_1
C^\mathbf{j}_{k-r,\ell}(x)+{_n\!Y}_{m-1}^\mathbf{i}\smile_1
C^\mathbf{j}_{k,\ell-m}(x)
-
\right.\\
\hspace{2.1in}
\left.
C_{k-r,\ell}^\mathbf{j}(x)\smile_1 {_n\!Y}^\mathbf{i}_{r-1}-
C^\mathbf{j}_{k,\ell-m}(x)\smile_1 {_n\!Y}^\mathbf{i}_{m-1}-\right.
\\
\left.
C^\mathbf{i}_{r,m}(x)\cdot C^\mathbf{j}_{k-r,\ell-m}(x)\,\,\right),
    \end{multline}
     and
   for $x\in P_n^{ev}:$
    \begin{multline}\label{C-kl-evn}
   d C^\mathbf{s}_{k,\ell}(x)   =
   \alpha^{1,1}_{k,\ell}(x)\cdot {_n\!Y}^\mathbf{s}_{k+\ell-1}-
   \sum_{\mathbf{i}_k\,\cup\,\mathbf{j}_\ell=\mathbf{s}}
   \epsilon_{k,\ell} \cdot {_n\!Y}^{\mathbf{i}_k}_{k-1}\sqcup_2
   {_n\!Y}^{\mathbf{j}_\ell}_{\ell-1}
   +
   \\
   \hspace{0.5in}\sum_{\mathbf{i}\,\cup\,\mathbf{j}=\mathbf{s}}
      \sum_{\substack{0< r<k \\ 0< m<\ell}}
\left((-1)^{\varsigma}\alpha^{*,*}_{r,m}(x)\cdot
{_n\!Y}_{r+m-1}^\mathbf{i}\Cup_1 C^\mathbf{j}_{k-r,\ell-m}(x)+\right.\\
\left.
\hspace{1.4in}
(-1)^{\varsigma_1} {_n\!Y}_{r-1}^\mathbf{i}\Cup_1
C^\mathbf{j}_{k-r,\ell}(x)+
(-1)^{\varsigma_2} {_n\!Y}_{m-1}^\mathbf{i}\Cup_1
C^\mathbf{j}_{k,\ell-m}(x)
-
\right.\\
\hspace{1.4in}
\left.
(-1)^{\varsigma_1}  C_{k-r,\ell}^\mathbf{j}(x)\Cup_1
{_n\!Y}^\mathbf{i}_{r-1}-
(-1)^{\varsigma_2}
{C}^\mathbf{j}_{k,\ell-m}(x)\Cup_1 {_n\!Y}^\mathbf{i}_{m-1}-\right.
\\
\left.
 C^\mathbf{i}_{r,m}(x)\cdot C^\mathbf{j}_{k-r,\ell-m}(x)\,\,\right),
    \end{multline}
     where
            $\varsigma=(k+\ell+1)(r+m),\,\varsigma_1=(k+\ell+1)r,\,
    \varsigma_2=(k+\ell+1)m;$
       in particular,  for $\mathbf{s}=(s_1)=(k+\ell-1)$ equalities (\ref{C-kl-od1}) -- \ref{C-kl-evn}) reduce to
 \begin{equation}\label{bc}
dc^{1,1}_{k,\ell}(x)= \alpha^{1,1}_{{k,\ell}}(x)\,y_{k+\ell-1} ,
 \end{equation}
 where $\alpha^{1,1}_{k,\ell}(x)$ is defined in subsection \ref{derhomotopy}.
Then  denoting
            \[
    C^*_{k,\ell}(x)=\sum_{\mathbf{s}\in \overline P(k+\ell)}\!\!
    C^\mathbf{s}_{k,\ell}(x),\]
       define
           for $x\in P^{od}_1:$
 \begin{multline}\label{pertu-bod}
   h(b_{k,\ell}(x))=-
x_{\ell-1}\sqcup_2 {Y_{k-1}}+\\
\sum_{\substack{ 0<r<k\\0<m<\ell}}
\left( \binom{r+m}{r}
x_{r+m-1}\smile_1 C^*_{k-r,\ell-m}(x)+\right.\\
\hspace{1.2in}
\left. b_{k-r,\ell}(x)\smile_1 {Y_{r-1}}+
b_{k,\ell-m}(x)\smile_1 {Y_{m-1}}+ \right.\\
\left.b_{r,m}(x)\cdot h({b}_{k-r,\ell-m}(x))\right)-
  C^*_{k,\ell}(x)+ p\, h(\tilde{b}_{k,\ell}(x)),
    \end{multline}
  \noindent  for $x\in P^{od}_{n>1}$ (i.e., $p=2$):
 \begin{multline}\label{pertu-bod2}
   h(b^{1,1}_{k,\ell}(x))=-
x_{\ell-1}\sqcup_2 {_n\!Y_{k-1}} +\\
\sum_{\substack{ 0<r<k\\0<m<\ell}}
\left(   X_{r-1}\star X_{m-1}
\smile_1 \!C^*_{k-r,\ell-m}(x)\,+
\right.\\
\hspace{1.2in}
\left. b^{*,*}_{k-r,\ell}(x)\smile_1 {_n\!Y_{r-1}}\,+
b^{1,1}_{k,\ell-m}(x)\smile_1 {_n\!Y_{m-1}}+
\right.\\
\left.b^{1,1}_{r,m}(x)\cdot h({b}^{*,*}_{k-r,\ell-m}(x))\right)-
  C^*_{k,\ell}(x)+ p\, h(\tilde{b}_{k,\ell}(x))
    \end{multline}
  and   for  $x\in P^{ev}_n:$
    \begin{multline}\label{pertu-bev}
   h(b^{1,1}_{k,\ell}(x))=-
x_{\ell-1}\sqcup_2  {_n\!Y_{k-1}}+\\
\sum_{\substack{ 0<r<k\\0<m<\ell}}
\left(\, (-1)^{\varsigma}
X_{r-1}\star X_{m-1}
\Cup_1 C^*_{k-r,\ell-m}(x)-
\right.\\
\hspace{1.2in}
\left.(-1)^{\varsigma_1} b^{*,*}_{k-r,\ell}(x)\Cup_1 {_n\!Y_{r-1}}-
(-1)^{\varsigma_2}b^{1,1}_{k,\ell-m}(x)\Cup_1{ _n\!Y_{m-1}}+ \right.\\
\left.b^{1,1}_{r,m}(x)\cdot h({b}^{*,*}_{k-r,\ell-m}(x))\,\right)-
  C^*_{k,\ell}(x)+ p\, h(\tilde{b}_{k,\ell}(x)).
    \end{multline}
       Note  that  $Y_0=Y_1=0$ for any $x\in P^*_n.$ In particular, for $n=1$
the equality $d^2 (C^\mathbf{s}_{k,\ell}(x))=0$ relies on the following
binomial equality for $s,t\geq 1$
    \[  \binom{s+t}{s} =
                \underset{0\leqslant r <s }{\sum} \binom{ s }{r}
                \binom{t}{s-r}+1.
      \]
      When $\tilde{b}_{k,\ell}(x)\neq 0,$
       we can choose $h$ similarly to $h(\tilde{x}_k)$ with
        $h^{tr}(\tilde{b}_{k,\ell}(x))=0.$
         Hence,  $c^{-1}_{k,\ell}(x)=0$  and $c_{k,\ell}(x)=c^0_{k,\ell}(x).$
         Furthermore, similarly to $h(Y_k),$  we have  $h(C^{\mathbf{s}}_{*,*}(x))=0;$   indeed, consider $b_{k,\ell}(x)$ such that $hb_{k,\ell}(x)$ contains
         $C^{\mathbf{s}}_{*,*}(x)$  with
    $ \mathbf{s}=(s_1,s_2,s_3)$  to be  of length $3.$ We have
     $dC_{k,\ell}(x)=(hd+hh)(b_{k,\ell}(x)).$ Then the cocycle
     $ (\varphi \circ \rho^{-1})((hd+hh)(b_{k,\ell}(x)))$
     is cohomologous to zero in $C^*(X;\mathbb Z).$  Consequently,
      $h(C^{\mathbf{s}}_{*,*}(x))=0$   in $RH.$
      In view of   $(\ref{C-kl-od1})-(\ref{C-kl-evn})$ 
      we can inductively   on the length of $\mathbf{s}$ achieve  $h(C^{\mathbf{s}}_{*,*}(x))=0$ 
     for any  $\mathbf{s}.$  


The first equalities  of (\ref{pertu-bod}) for $x\in P^{od}_1$ with
$\tilde{b}_{k,\ell}(x)=0$  read:
\[
\begin{array}{llll}
h(b_{11})=-c_{11} \ \  \text{with}\ \ dc_{11}=0,  \\
h(b_{12})= x_0\smile_1 c_{11}-c_{12}  \ \  \text{with}\ \ dc_{12}=3y_2,
 \vspace{1mm}  \\
h(b_{13})= x_0\smile_1 c_{12}+x_1\smile_1 c_{11}- c_{13}  \ \
\text{with}\ \ dc_{13}=4 y_3, \vspace{1mm} \\
h(b_{14})= x_0\smile_1 c_{13}+x_1\smile_1 c_{12}
-x_2\smile_1 c_{11}+ b_{11}\smile_1 y_3  -c_{14}  \\  
 \hspace{3.7in}  \text{with}\ \ dc_{14}=5 y_4, \vspace{1mm}
\\
h(b_{22})=x_0\smile_1\!(c_{12}+c_{21})   +2x_1\!\smile_1\! c_{11}+
b_{11}c_{11} - C^{11}_{22} -c_{22}\vspace{1mm}\\ 
\hspace{2.5in} \text{with} \ \,
d(C^{11}_{22})=c_{11}c_{11} \ 
  \text{and}\ dc_{22}=6y_3,\vspace{1mm}
\\
h(b_{33})=x_0\smile_1\!(C^*_{23}+C^*_{32})   +
x_1\!\smile_1\! (2C^*_{22}+C^*_{13}+C^*_{31})+
x_2\!\smile_1\! 3(c_{12}+c_{21})+\vspace{1mm}\\
 \hspace{0.6in} b_{11}\cdot hb_{22}+ b_{12}\cdot hb_{21}+ b_{21}\cdot
 hb_{12}+b_{22}\cdot c_{11}- C^{13}_{33}-C^{22}_{33}-c_{33} \ \
 \text{with}       \vspace{1mm}\\
 d(C^{\mathbf{s}}_{33})=\begin{cases}
 6y_3\smile_1 c_{11}-c_{11}c_{22}-c_{22}c_{11}, & \mathbf{s}=(1,3),\\
 18y_2\sqcup_2
 y_2+3y_2\smile_1(c_{12}+c_{21})-c_{12}c_{21}-c_{21}c_{12}, &
 \mathbf{s}=(2,2),\\
 5y_5, &\mathbf{s}=(5).
  \end{cases}
\end{array}
\]

\noindent  On $\mathcal{U}:$ When $d(a\sqcup_2 b)$ agrees with the
$d(a\smile_2b)$   we have $h^{tr}(a\sqcup_2 b)=0.$
While
$h^{tr}(a\sqcup_2 b)$ may be  non-trivial when $a=b.$
Namely, let $x=[t_2(x_0)]\in  H^*(X;\mathbb Z_2)$  with
 $x_0\in \mathcal{V}^{0,\ast}\cup\mathcal{O}^{-1,\ast}_2.$
Then
$h(x_0\sqcup_2 x_0)=h^{tr}(x_0\sqcup_2 x_0)$ and
 \[  [t_2(h(x_0\sqcup_2 x_0))] = Sq_1(x)\in H^*(X;\mathbb Z_2).\]
 Consequently,
$ h(x_0\sqcup_2 x_0 )\neq 0$  when  $Sq_1(x)\neq 0.$
The  value of $h$ on $x_0\sqcup_2 \cdots \sqcup_2 x_0$   may be also
non-zero, but this is not important in the sequel.

\section{The Secondary cohomology operations
$\psi_{r,n}$}\label{2operations}

To construct the secondary cohomology operations  we use  the minimal
filtered  Hirsch  model (\ref{zig-zag}) of  $C^*(X;\mathbb{Z}).$ Note
that the basic equalities   for the  construction  of  $\psi_{r,n}$ 
 in  $C^*(X;\mathbb{Z})$ that correspond to that in $RH$
hold only up to homotopy,
so that
 the immediate construction of $\psi_{r,n}$     in the integral
 simplicial (cubical)
 cochain complex of $X$
requires to evoke  the all canonical multilinear  Hirsch structural
operations $\{E_{p,q}\}_{p,q\geq 1}$  beside the $\smile_1$ --  product
(compare (3.7)--(3.8) in \cite{saneFiltered}).
 In other words, the  filtered
  Hirsch  model $(RH,d_h)$ allows us to construct the secondary
  cohomology operations by rather simplified formulas.
   For $s\geq 1,$ denote
   \[
   (k_s,\ell_s)=  \begin{cases}
       (p^{s}, (p-1)p^{s})  , &  x\in P^*_1(X),\\
        (2p^{s-1}, 2(p-1)p^{s-1}) , & x\in P^*_{n>1}(X).
       \end{cases}
       \]
  Tacking into account    (\ref{split})   and (\ref{trans})
    define
  the following cohomology elements:
  
  \noindent For $p>2$  and  $r\geq 1:$
\begin{equation}\label{psiprime}
\begin{array}{llll}
  \psi^{\prime}_{r,1}(x)=[t_p(c_{k_r,\, \ell_r}(x))]=
  [\, t_p(d_{h'}b_{k_r,\,\ell_r}(x))\,]     \in H^{(m-1)p^{r+1}+1}(X;\mathbb Z_p),
  \vspace{1mm} \\
   \hspace{3.56in} 
      x\in P^m_1(X),
     \vspace{2mm}\\
 \psi^{\prime}_{r,n}(x)=[t_p(c_{k_r,\, \ell_r}(x))]=
[\,t_p(d_{h'}b_{ k_r,\,\ell_r}(x))\,]
\in  H^{(m(n+1)-2)p^r+1}(X;\mathbb Z_p), \vspace{1mm}\\
\hspace{3.55in}  x\in P^m_{n>1}(X),
\end{array}
\end{equation}
and for $p=2$  and  $r\geq 1:$
\begin{equation}\label{psiprime2}
\begin{array}{llll}
  \psi^{\prime}_{r,n}(x)=[t_2(c_{{2^r},\, 2^r}(x))]=
[\,t_2(d_{h^{\prime}}(b_{ 2^{r},\,2^{r}})(x) )\,]\in\vspace{1mm}\\
\hspace{2.0in}
H^{(m(n+1)-2)2^r+1}(X;\mathbb Z_2),\,\, x\in P^m_{n}(X).
\end{array}
\end{equation}
Given $ x\in P^\ast_{n}(X),$ apply to  (\ref{zig-zag})  and choose
$\varphi\varrho^{-1}( c_{k_r,\,\ell_r}(x))\in  C^*(X;\mathbb{Z}) .$
Then the mod $p$ cohomology classes of two such choices
differ by the value of a primary  operation  on $x,$ and, hence,
(\ref{psiprime})--(\ref{psiprime2}) induces the  secondary cohomology
operations
for $p\geq 2$ and $r\geq 1$
\begin{equation}\label{secondary2}
\begin{array}{rrl}
\psi_{r,1}:P^m_1(X) & \rightarrow & H^{(m-1)p^{r+1}+1}
(X;\mathbb Z_p)/\operatorname{Im}\mathcal{P}_1,
\vspace{1mm}\\
\psi_{r,n}: P^m_{n>1}(X) & \rightarrow  & H^{(m(n+1)-2)p^r+1 }(X;\mathbb Z_p)/
\operatorname{Im}\mathcal{P}_1,
\end{array}
 \end{equation}
 that are linear for  $n+1=p^k,\,k\geq 1.$
 Note that
 $\psi_{1,p^k-1}$ coincides with   the Adams secondary cohomology
 operation $\psi_k$ for
$p$ odd or $p=2$ and $k>1$ (cf. \cite{adams-hopf1}, \cite{kraines2},
\cite{harper}).

\subsection{The generators $\omega_{r,n}(x)$}

To calculate  $H^\ast(t_p(\overline{V}),\bar d_h)$ consider
 the cochain complex $(\overline{V},\bar d_h).$ 
 We have that only the elements
 $\bar x_m \in \overline{\mathcal{X}}_p\subset \overline{V}$ are $\bar
 d$-cocycles  mod $p$ for all $m,$ while
 from (\ref{h-odd2})--(\ref{h-even2}) we deduce that
  $\bar x_m$  is  $\bar d_h$-cocycle   mod $p$ if and only if $y_{k}\in
  \mathcal{D}^\ast_p(RH)$
 for all $k\leq m$ ($\mathcal{D}^\ast_p(RH)$  is given by (\ref{decompos}))
 in which case $\bar x_m$ is
  automatically non-cohomologous to zero unless $m=0,$ and, hence,
defines the non-trivial cohomology class
$  [t_p(\bar x_m)]\in  H^*\left(t_p(\overline{V}) ,\bar d_h\right)$
for $m>0.$
Let $\epsilon_x$ be given by (\ref{ex}),
and let $\ell_x\geq \epsilon_x$ be
  the smallest integer
 such that
 \begin{equation}\label{ellx}
   y_{m}\in \mathcal{D}^\ast_p(RH) \ \
  \text{for} \  \left\{\!\!
  \begin{array}{ll}
  m< p^{\ell_x+1}-1, &
    x \in P^{od}_1(X),\vspace{1mm}\\
  m< 2p^{\ell_x}-1, &
 \text{otherwise}
 \end{array}
 \right.
 \end{equation}
 (when $p=2,$ the both items above  are the same).
  Then for $1\leq r\leq \ell_x$ define
  $\omega_{r,n}(x)\in   H^*(\Omega X;\mathbb Z_p)$
  by
  \begin{equation}\label{omega-r}
\omega_{r,n}(x)=\begin{cases}
         \left[\,{t_{p}( \bar x_{p^r-1})}\,\right] , & x\in
         P^{*}_1(X), \vspace{1mm}
           \\
           \left[\, {t_{p}( \bar x_{2p^{r-1}-1})}\,\right] ,  &
            x\in P^*_{n>1}(X);
         \end{cases}
\end{equation}
in particular, $\omega_{1,1}(x)=[t_p(\bar x_{p-1})]$ for $n=1$ and
$\omega_{1,n}(x)=[t_p(\bar x_{1})]$ for $n>1.$ From (\ref{syzygies1}),\,(\ref{cup1monom}),\,(\ref{h-odd1}) and (\ref{hirsch1}),\,(\ref{quadratic}) follows that
$\omega_{r,n}(x)$ satisfy (\ref{coprod}); in particular, the  summand $a'_r\otimes a''_r$
in (\ref{coprod}) is determined by the quadratic components of the perturbation $h$ in
$(RH,d_h).$

The following proposition is starting point to relate the secondary cohomology
operations with the loop space cohomology.

 Given $r\geq 1,$ denote
 \[  \mathbf{p}_m=
 \begin{cases}
  (p^{m+1}\!-1,...,p^{m+1}\!-1) \in P_{p^{r-m}}\,(\,p^{r+1}\!-p^{r-m}) ,&x\in P^*_1(X),\, m\geq 0,\\
  \hspace{0.16in} (2p^m\!-1,...,2p^m\!-1) \in P_{p^{r-m}}\,(\,2p^r\!-p^{r-m}) ,&
   x\in P^*_{n>1}(X),\, m\geq 1,
   \end{cases}
   \]
   and let $ \mathbf{p}_{m,i}$ denote the subsequence of  $\mathbf{p}_{m}$
   of length $i.$
   
   \begin{proposition}\label{first}
 Given  $x\in P^\ast_n(X)$  and   $1\leq r\leq \ell_x,$
  $d_{h'}(b_{k_r,\,\ell_r}(x))$ may contain only the multiplicative generators 
  of the form $x_{k_r-1}\smile_1 x_{\ell_r-1}$  and\, $C^{\mathbf{p_m}}_{k_r,\ell_r}(x)\!\mod p$  
  with
 \[ [t_p(\,\overline{x_{k_r-1}\!\smile_ 1\! x_{\ell_r-1}}\,)]=
  \sum_{\epsilon_x\leq m\leq r}
  \left[t_p\left(\,\overline{C}\,^{\mathbf{p}_m}_{k_r,\,\ell_r}(x)\right)\right] .
   \]

\end{proposition}
  \begin{proof}
  Consider  equalities  (\ref{C-kl-od1})--(\ref{C-kl-evn}) for $dC_{k_r,\,\ell_r}^{\mathbf{s}}$  with  $\mathbf{s}=\mathbf{p}_m.$ In view of  (\ref{ellx})  the components $Y^*_* \mod p$ vanish therein.
 Also, when
     $\alpha_{k,\ell}(x)$ is not divisible by $p$  in (\ref{b-kl}) we can set
    $c_{k,\ell}(x)=0\mod p$ (cf. Remark \ref{monom}).
 Consequently, for
  $\mathbf{s}$ being  different from $\mathbf{p}_{m}$  we can take $C_{k_m,\ell_m}^{\mathbf{s}}(x)=0\mod p,$ 
   and then  for $\epsilon_x\leq m\leq r$
 \begin{equation}\label{C1}
   dC^{\mathbf{p}_{m}}_{k_r,\ell_r}(x)=  
    \sum_{\substack{i+j=p^{r-m}\\ i,j\geq 1}}
   C^{\mathbf{p}_{m,i}}_{ik_m,\,i\ell_m}(x)\,C^{\mathbf{p}_{m,j}}_{jk_m,\,
   j\ell_m}(x) \mod p.
        \end{equation}
       Checking the divisibility of  $\alpha_{k,\ell}(x)$ by $p$
  in  (\ref{pertu-bod}) -- (\ref{pertu-bev}) finishes the proof.
 
 \end{proof}

\begin{proposition}\label{smilepowers}
For $p\geq 2,$
\[
\left[ t_p\left(\,\overline{x_{k_r-1}^{\smile_1 p}}\,\right)\right]=
\sum_{\epsilon_x\leq m\leq r}
 \left[t_p\left(\overline{C}\,^{\mathbf{p}_{m}}_{k_r,\ell_r}(x)\right)\right].
 \]

\end{proposition}
  \begin{proof}
  For $p=2$ nothing is to prove. Let $p>2.$
   Given $k\geq  1$  and $\{b_{k,\ell}(x)\}_{k,\ell\geq1},$ form the sequence $\{\mathfrak{b}
   _{k,\,qk}(x)\}_{1\leq q< p }$ in $RH$ as
follows. Set
\[
\hspace{-1.8in} \mathfrak{b}_{k,k}(x):=b_{k,k}(x),\ \    q=1,\ \ \text{and} \]
\[
 \mathfrak{b}_{k,qk}(x): =\alpha_{k, (q-1)k}(x)\, {b}_{k,qk}(x)-
  x_{k-1}\!
  \smile _{1}\!  \mathfrak{b}^{q-1}_{k,(q-1)k}(x),\ \, q>1     .
\]
     Then $d_{h'}\left(\mathfrak{b}^{p-1}_{k,(q-1)k}(x)\right)$ for $q=p$ 
     and  $k\in \{p^r, 2p^{r-1}\} $
     contains the multiplicative generators $-x_{k-1}^{\smile_1 p}$ and
    \[\sum_{\epsilon_x\leq m \leq r}\alpha_{k,(p-2)k}(x)\, C_{k,(p-1)k}^{\mathbf{p}_{m}}(x).   \]
    Since
    $\alpha_{k,(p-2)k}(x)=\!\binom{p^{s+1}-p^s}{p^s}=-1\mod p$ for
    $k\in \{p^s, 2p^{s}\},\,s\geq0,$ proposition follows.
  \end{proof}

  \begin{proposition} \label{cw}
 For $\epsilon_x\leq m\leq r$ the element  $\left[t_p\left(\overline{C}\,^{\mathbf{p}_m}_{k,\ell}(x)\right)\right] $ is identified
  as
  
  \[
   \left[t_p\left(\overline {C}\,^{\mathbf{p}_{m}}_{k_r,\ell_r}(x)\right)\right]=
  \omega_{r-m,n}\left(\psi_{m,n}(x)\right).
  \]
  
\end{proposition}
  \begin{proof}
 
   Consider (\ref{C1}), and  for $m'=p^{r-m}$ denote
   \[ 
   x'_{m'}:= C^{\mathbf{p}_{m}}_{k_r,\, \ell_r}(x)
   \]
   and
   \[
   x'_{i-1}:= C^{\mathbf{p}_{m,i}}_{ik_m,\, i\ell_m}(x)
    \ \  \text{with}\ \
    x'_0:=C^{\mathbf{p}_{m,1}}_{k_m,\, \ell_m}(x),
      \]
    and then rewrite  (\ref{C1})  as  
   \[ dx'_{m'}=
  \underset{\substack{s+t=m'-1\\s,t\geq 0}}{\sum}  
  x'_{s}x'_{t}.
  \]
  Taking into account (\ref{trans})  the definition
  (\ref{secondary2})   implies that
   $x'_0=\psi_{m,n}(x)$ for $m>0,$ while
  a straightforward calculation shows that
  $x'_0=\mathcal{P}_1(x)$ for $m=0$ (i.e., $x\in P^*_1(X)$). Thus,
   $x'_0=\psi_{m,n}(x)$ for $m\geq 0$ (cf. (\ref{psi0})).
  Set $x'_{m'}$ instead of $x_m$ in (\ref{omega-r}) to obtain the equality of the proposition.
   \end{proof}
 For $ p=2 $  and  $(k,\ell)=(2,2)$    see  Example \ref{psi11} below.

\section{The proofs of Theorems \ref{premain} and \ref{sigmazero}}

\noindent\emph{The proof of Theorem \ref{premain}.}

(i)  The proof follows from (\ref{susp}).

(ii) By definitions (\ref{secondary2}), (\ref{omega-r}) and Propositions   \ref{smilepowers} - \ref{cw} follows
\[
 \omega_{r,n}(x)^{p}= 
\sum_{\substack{i+j=r\\ i, j\geq 0}} \!\!   \omega_{i,n}(\psi_{j,n}(x)), 
\]
and  taking the  $p^{k-1}$- th powers on the both sides of the equality
finishes  the proof.

\hspace{4.6in}    $\Box$

\medskip

\noindent\emph{ The proof of Theorem \ref{sigmazero}}. Given $x\in P^*_n(X),$  the definition of $\ell_x$ in (\ref{ellx}) implies that
\begin{equation}\label{m}
   y_{m}\notin \mathcal{D}^\ast_p(RH) \ \
  \text{for} \  m= \left\{\!\!
  \begin{array}{ll}
   p^{\ell_x+1}-1, &
    x \in P^{od}_1(X),\vspace{1mm}\\
  2p^{\ell_x}-1, &
 \text{otherwise};
 \end{array}
 \right.
 \end{equation}
consequently, if $z:=[t_p(y_m)],$ then $ z\in\operatorname{Ker}\sigma. $
The definition of  $\alpha^{1,1}_{k,\ell}(x)$ in  (\ref{bc}) yields
\[
\beta_{\ell_x+\nu_x} [t_p(c_{k,\ell}(x))]= z\]
for
\[
(k,\ell)=
\left\{
  \begin{array}{lll}
(1, p^{\ell_x+1}-1)  , & x \in P^{od}_1(X), \vspace{1mm} \\
   
  (2, 2(p^{\ell_x}-1))   , & \text{otherwise}
  \end{array}\ \
  \right.
  \]
  and
 \[ \nu_x=
\left\{
  \begin{array}{lll}
1  , & x \in P^{od}_1(X), \vspace{1mm} \\
   
0   , &   \text{otherwise}.
  \end{array}
  \right.
\]
A straightforward calculation 
 yields for $\epsilon_x\leq r\leq \ell_x$
\begin{equation*}
  [t_p(c_{k,\ell}(x))] =\left\{
 \begin{array}{llll}
  \mathcal{P}^{(r+1)}_1(x) ,&     x\in
 P^{od}_1(X), &  (k,\ell)=(1, p^{r+1}-1),
   \vspace{2mm}\\
 \mathcal{P}^{(r-1)} \psi_{1,n}(x)  ,    &
 \,
   x\in  P^\ast_{n}(X)\setminus P^{od}_1(X),&  (k,\ell)=(2, 2(p^{r}-1)).
  \end{array}
  \right.
\end{equation*}
Thus,
\begin{equation}\label{beta-r}
\begin{array}{rl}

 \beta_{\ell_x+1}\mathcal{P}^{(\ell_x+1)}_1(x) = z,
                                                & x\in P^{od}_1(X),
  \vspace{1mm}\\
 \beta_{\ell_x}\mathcal{P}^{(\ell_x-1)}_1\psi_{1,n}(x)=z  ,   &
 x\in P^\ast_n(X)\setminus P^{od}_1(X).
  \end{array}
\end{equation}
Conversely,  if $z\in \operatorname{Ker} \sigma\cap \mathcal{H}_p(X),$ then  $z=\rho(y_m)$   with $m$  defined in (\ref{m})  for some $x\in P^*_n(X),$ and, hence, $z\in \mathcal{I}^*_p (X).$

\hspace{4.6in}    $\Box$

\begin{example}\label{psi11}
 Let $p=2,$     and  consider the construction of
 $\psi_{1,1}(x)$   for  $x\in P^{ev}_1(X).$ Let $x=[t_2(x_0)]$ with  $dx_1=-x_0^2$ in
 $RH.$
 Denoting simply $b_{ij}:= b_{ij}(x), $ we have the equalities
 \[  db_{11}=x_0\smile_1 x_0, \]
  \[db_{12}=x_2+x_0\smile_1 x_1 -  b_{11}x_0- x_0b_{11},\]
  \[db_{21}=x_2+x_1\smile_1 x_0 +  b_{11}x_0+ x_0b_{11}\]
  and
   \[db_{22}=2x_3+x_1\smile_1 x_1 -
    (x_0\smile_1 x_0)b_{11}+ (b_{12}+b_{21})x_0 + x_0
    (b_{12}+b_{21}).\]
      From the second equality resolve
  $x_2=-x_0\smile_1 x_1+ b_{11}x_0+x_0b_{11}$   with $b_{12}=0,$
  and set it in the third one to obtain
   \[db_{21}=-x_0\smile_1 x_1+x_1\smile_1 x_0 +  2(b_{11}x_0+
   x_0b_{11}).\]
  Thus, $h(x_2)= h'(x_2)   =-c_{11}x_0-x_0c_{11}$ for $-c_{11}=hb_{11}$
     with  $[t_2(c_{11})]=Sq_1(x).$
      For an even dimensional cochain $a\in C^{ev}(X;\mathbb Z)$
     the equality
           $d(a\smile_2 a)=2 \,a\smile_1 a$
           implies $2[a\smile_1 a]=0 \in H^*(X;\mathbb Z).$ This means that
            there is a generator $e\in R^{-1}H$ with
            $de = 2c_{11}.$
     Then denoting
     \[
     e_1:= - x_0\smile_1c_{11}- e\,x_0,
     \] obtain
     $de_1=h(x_2).$
      Thus, $d^2_h(b_{21})=0$ yields
  \[dh(b_{21})=2h(x_2) \ \  \text{for}\ \ hb_{21}=2e_1.\]
   Recall that
  $dx_3=-x_0x_2+x_1x_1-x_2x_0.$ Then $d^2_h(x_3)=0$ implies
  \[h(x_3)= h'(x_3)+ h^{tr}(x_3)=  x_0e_1+
  e_1x_0+ y_3;\]
       in particular,  denoting $\tilde x_2:=x_2 - e_1$
        obtain $d_h(\tilde{x}_2)=-x_0x_1+x_1x_0,$ and then
        the cohomology class
        \[
       z:=[t_2(y_3)]= [t_2(x_0x_2-x_1x_1+ x_2x_0 - h'(x_3))=[t_2(x_0\tilde
       x_2-x_1x_1+\tilde x_2x_0)]
       \]
       determines  the symmetric Massey product
       $\langle x \rangle^4.$

    Consider the composition $hd(b_{22}).$ From the equalities
    $b_{12}=0$ and
    $2h'(x_3)=2(x_0e_1+ e_1x_0)= x_0 h(b_{21})+h(b_{21})x_0$ follow that
     \[hd(b_{22}) = -(x_0\smile_1 x_0)c_{11}.\]
     But also $d(b_{11}c_{11})=(x_0\smile_1 x_0)c_{11},$ and, consequently,
     \[
      hdb_{22}+ d_h(b_{11}c_{11})=-c_{11}c_{11}.
     \]
     Hence, $[t_2(c^2_{11})]=0\in H^*(X;\mathbb Z_2).$
     This means that there is a generator $C^{11}_{22}\in R^{-1}H$
     with $dC^{11}_{22}=- c_{11}^2.$
     Now
          $d^2_h(b_{22})=0$ implies that $h$ is defined on $b_{22}$ as
  \[
 hb_{22}=  b_{11}c_{11}- C^{11}_{22} -c_{22}\ \ \text{with}  \ \
   dc_{22}=2y_3.
   \]
    By definition
  \begin{multline*}
  \psi_{1,1}(x)=[t_2(c_{22})]=
  [t_2\left(x_1\smile_1 x_1-
  (x_0\smile_1x_0)b_{11}+ b_{21}x_0 +x_0b_{21}+ \right.\\
  \left.
  b_{11}c_{11} -C^{11}_{22}\right) ].
  \end{multline*}
Consequently,  $ dc_{22}=2y_3$ implies
  \[ \beta \psi_{1,1}(x)= z \,\,(=\langle x \rangle^4) .\]
  Thus, $\sigma z=0$ and
   $\ell_x\geq 1.$
Furthermore,  for $w:=-x_1\smile _1 x_1-x_0b_{21}-b_{21}x_0   $    we have  the equality
    \[d_h(w)= -2(x_0\tilde x_2-x_1^2+\tilde x_2x_0)+(x_0\smile_1 x_0)^2,
   \]
    and then $d_h(-2w +x_0^{\smile_1 4})=-4(x_0\tilde x_2-x_1^2+\tilde x_2x_0)$ in $RH.$
     Since $[t_2(   -2w +x_0^{\smile_1 4})]= [t_2(x_0^{\smile_1 4})] =Sq^{(2)}_1(x)\in H^*(X;\mathbb Z_2),$ obtain
       $\beta_2 Sq^{(2)}_1(x)=z $  (cf. \cite{kraines2}).

       Finally, consider \[\omega_1(x)=[t_2(\bar x_1 )]\ \ \text{and}\ \
       \omega_1(Sq_1(x))=[t_2(\bar C^{11}_{22} )]\ \
        \text{in} \ \ H^*(\Omega X;\mathbb Z_2),
       \]
        and obtain
       \[    \omega^2_1(x)=\sigma\psi_{1,1}(x)+\omega_1(Sq_1(x))  .\]
  \end{example}

\hspace{4.6in}    $\Box$

\section{The mod $p$  loop  cohomology  ring of the exceptional  group
$F_4$}

 The mod $p$  cohomology ring of the exceptional  group $F_4$ is free
 unless $p=2,3.$ Here we just consider these cases and use
 Theorem \ref{main} for calculation the loop cohomology ring of $F_4$
 (cf. \cite{Toda}, \cite{watanabe}). In view of the layout of multiplicative generators in the cohomology ring these calculations do not  need in fact to involve the cochain complex of $F_4.$

1. Let $p=2.$ It is well known that for   $x_i\in   H^i(F_4;\mathbb Z_2)$
\[
H^*(F_4;\mathbb Z_2)=\mathbb Z_2[ x_3]/(x_3^4)\otimes \Lambda(x_5, x_{15},
x_{23})
\]
with $Sq_1x_3=x_5.$   We have the
following  sequences in the Hirsch  resolution $(RH,d)$ of  $H^*=H^*(F_4;\mathbb  Z_2)$
\[\{ (x_{3,k})_{k\geq 0},\, (x_{5,k})_{k\geq 0},\,  (x_{15,k})_{k\geq
0},\,(x_{23,k})_{k\geq 0}\}
\]
with $dx_{3,1}=x^4_{3,0}$  and $dx_{i,1}=x^2_{i,0}$ for $i=5,15,23.$
Let $\mathcal{H^*}$ be the set of multiplicative generators of $H^*.$
Tacking into account $\mathcal{H}^{ev}=0$  from
(\ref{syzygies1}),(\ref{syzygies2}), (\ref{h-odd1})--(\ref{h-even2}) follows that
$h(x_{i,k})$  is decomposable  in the filtered Hirsch model $(RH,d_h)$, and, hence,
$ [\bar x_{i,k}]$  with
$ [\bar x_{i,0}]=\sigma( x_i)$  is non-zero in $ H^*(\Omega F_4;\mathbb  Z_2)$
   for all $i\in \{3,5,15, 23\}$  and $k\geq 0.$
 The equality
$Sq_1x_3=x_5$
 implies $   \sigma( x_{3})^2=\sigma( x_{5}).  $
 Since $Sq_1^{(2)}(x_3)\in H^9$ and $\mathcal{H}^{9}=0,$ we have
$ \sigma (x_{3})^4=  \sigma (x_{5})^2 =0;$
from
$Sq_1(x_{15})\in  H^{29}, $ $Sq_1(x_{23})\in  H^{45} $ and $\mathcal{H}^{29}=\mathcal{H}^{45}=0$ follows that 
$ \sigma( x_{i})^2=0$  for $i=15,23.$

Recall that 
$\psi_{r,n}(x_i)\in H^{(i(n+1)-2)2^r+1} $ for  $n\geq1.$
From    $\psi_{r,3}(x_3)\in H^{10\cdot 3^r+1} $ and 
$\mathcal{H}^{10\cdot 3^r+1}=0$  follows
 $[\bar x_{3\,,\,2^r-1})]^2=0$ for  $r\geq 1;$ 
 similarly, from    $\psi_{r,1}(x_i)\cap \mathcal{H}^*=0$  follows
 $[\bar x_{i\,,\,2^r-1}]^2=0$ for $i\in \{  5, 15, 23\}$   and  $r\geq 1.$
  Consequently, for $\omega_{r,n_i}(x_{i})=[\bar x_{i\,,\, 2^r-1}]$
  with $n_3=3$ and $n_i=1$ for $ i\in
\{
5,15,23\},$ 
we obtain
\begin{multline*} H^*(\Omega F_4;\mathbb  Z_2)=
\mathbb Z_2[\sigma (x_3)]/\!\left(\sigma(x_3)^4\right)\otimes \\
\Lambda(\omega_{r,3}(x_3)_{r\geq 1},\, \omega_{ r,1}(x_5)_{ r\geq 1},\,\omega_{r,1}(x_{15})_{r\geq 0}),\,\omega_{ r,1}(x_{23})_{ r\geq 0}    )     .
\end{multline*}

2. Let $p=3.$     We have for $x_i\in   H^i(F_4;\mathbb Z_3)$
\[
H^*(F_4;\mathbb Z_3)=\mathbb Z_3[ x_8]/(x_8^3)\otimes \Lambda(x_3, x_7,
x_{11}, x_{15})
\]
with    $\mathcal{P}_1(x_3)=x_7$   and $\beta\mathcal{P}_1(x_3)=x_8.$
 There are the following sequences in the Hirsch resolution   $(RH,d)$ of
 $H^*=H^*(F_4;\mathbb  Z_3)$
\[\{ (x_{3,k})_{k\geq 0}, \,(x_{7,k})_{k\geq 0}, \, (x_{8,k})_{k\geq 0},\,
(x_{11,k})_{k\geq 0}, \, (x_{15,k})_{k\geq 0}\}
\]
with $dx_{8,1}=x^3_{8,0}.$ Let $\mathcal{H^*}$ be the set of multiplicative generators of $H^*.$
From the equality $\beta\mathcal{P}_1(x_3)=x_8$
and Theorem \ref{sigmazero} follows that $\sigma x_8=0;$
in particular,  $h^{tr}(x_{3,2})=x_8$ in the filtered Hirsch model $(RH,d_h)$ (cf. the proof of Theorem
\ref{sigmazero}), and, consequently, 
 from
(\ref{h-odd1})--(\ref{h-odd2}) we deduce that
$ \bar x_{3, k} $ is not $\bar d_h$ -- cocycle for $k>2,$ too;  hence, it does not produce an element in $H^*(\Omega F_4;\mathbb{Z}_3).$
Since $\mathcal{H}^{2m}=0$ unless $m= 4,$ from
(\ref{syzygies1}), (\ref{syzygies2}), (\ref{h-odd1})--(\ref{h-even2}) follows that
$h(x_{i,k})$  is decomposable,  and, hence,
$ [\bar x_{i,k}]$ is non-zero in $ H^*(\Omega F_4;\mathbb  Z_3)$
  with
$ [\bar x_{i,0}]=\sigma ( x_i)$  for all $i\in \{7,11, 15\}$  and $k\geq 0.$

 From the equality  $\mathcal{P}_1(x_3)=x_7$ follows
$  \sigma( x_{3})^3=\sigma (x_{7}) . $
  From   $\mathcal{P}_1^{(2)}(x_3)\cap \mathcal{H}^{19}=0$   follows
$ \sigma( x_{3})^9=  \sigma( x_{7})^3   =0;$ from
$\mathcal{P}_1(x_i)\cap \mathcal{H}^*=0$   follows
\[ \sigma( x_{i})^3=0\ \  \text{for} \ \ i\in\{11,15\}.\]
Recall that 
$\psi_{r,n}(x_i)\in H^{(i-1)3^{r+1}+1} $ for  $n=1,$ while
$\psi_{r,n}(x_i)\in H^{(i(n+1)-2)3^r+1} $ for  $n>1.$
From    $\psi_{r,2}(x_8)\in H^{7\cdot 3^r+1}$ and $\mathcal{H}^{7\cdot 3^r+1}=0$
  follows 
$[\bar x_{8\,,\, 2\cdot 3^{r-1}-1}]^3=0$ for $r\geq 1,$ and   from
$\psi_{r,1}(x_i)\cap \mathcal{H}^*=0$   follows
  $[\bar x_{i\,,\,3^r-1}]^3=0$ for   $i\in\{3,7,11,15  \}$ and $r\geq
  1.$
Consequently, for $\omega_{r,n_i}(x_{i})=[\bar x_{i\,, \,3^r-1}],\, i\in
\{3,7,11,15\}$ and $\omega_r(x_8,2)=   [\bar x_ {8\,,\,2\cdot 3^{r-1}-1}],$
we  obtain

\begin{multline*} H^*(\Omega F_4;\mathbb  Z_3)=\mathbb
Z_3[\sigma(x_3)]/\!\left(\sigma(x_3)^9\right)\otimes
\mathbb Z_3[\omega_{r,1}(x_7)]/\!\left(\omega_{r,1}(x_7)^3\right)_{r\geq 1}\otimes \\
\hspace{0.8in}
\mathbb Z_3[\omega_{ r,2}(x_8)]/\!\left(\omega_{r,2}(x_8)^3\right)_{r\geq1}\otimes
\mathbb Z_3[\omega_{r,1}(x_{11})]/\!\left(\omega_{r,1}(x_{11})^3\right)_{r\geq0}\otimes \\
\mathbb Z_3[\omega_{ r,1}(x_{15})]/\!\left(\omega_{r,1}(x_{15})^3\right)_{r\geq0}.
\end{multline*}

 \end{document}